\documentclass[11pt, reqno]{amsart}
\usepackage{enumitem}
\usepackage{stmaryrd}
\usepackage[paperheight=11in,paperwidth=8.5in,left=1.25in,top=1.25in,
right=1.25in,bottom=1.25in]{geometry}

\usepackage[utf8]{inputenc}

\usepackage{comment}
\usepackage[english]{babel}
\usepackage[T1]{fontenc}
\usepackage{todonotes}

\usepackage{amssymb}
\usepackage{bm}

\usepackage{amsmath}
\usepackage{enumitem}
\usepackage{amssymb,amsfonts,stmaryrd,mathrsfs,mathabx}
\usepackage{latexsym,amsthm,mathtools,bbm}
\usepackage{amsthm}
\usepackage{pgf,tikz} 
\usepackage{subcaption}
\usepackage[colorlinks,cite color=green,link color=purple]{hyperref}
\usepackage{todonotes}
\usepackage{cases}
\usepackage[thinlines]{easytable}
\usepackage{bm}
\usepackage[export]{adjustbox}
\usepackage{thmtools}
\usepackage[capitalize]{cleveref}
\newtheorem{thm}{Theorem}[section]
\newtheorem{cor}[thm]{Corollary}
\newtheorem{lem}[thm]{Lemma}
\newtheorem{prop}[thm]{Proposition}

\theoremstyle{definition}
\newtheorem{definition}[thm]{Definition}
\newtheorem{remark}[thm]{Remark}
\newtheorem{example}[thm]{Example}



\allowdisplaybreaks

\newcommand{\bfx}{\boldsymbol{x}}

\DeclareMathOperator{\MLQ}{MLQ}
\DeclareMathOperator{\SMLQ}{MLQ^{\pm}}
\DeclareMathOperator{\Tab}{Tab}

\DeclareMathOperator{\wt}{wt}
\DeclareMathOperator{\skipped}{skip}
\DeclareMathOperator{\free}{free}
\DeclareMathOperator{\emp}{empty}
\DeclareMathOperator{\hook}{hook}
\DeclareMathOperator{\classic}{cl}
\DeclareMathOperator{\primed}{pr}

\DeclareMathOperator{\negative}{neg}
\DeclareMathOperator{\coinv}{coinv}
\DeclareMathOperator{\maj}{maj}
\DeclareMathOperator{\arm}{arm}

\DeclareMathOperator{\leg}{leg}
\DeclareMathOperator{\ball}{ball}
\DeclareMathOperator{\pair}{pair}
\DeclareMathOperator{\rev}{rev}
\DeclareMathOperator{\Span}{Span}
\DeclareMathOperator{\Pack}{Pack}
\DeclareMathOperator{\Supp}{Supp}

\newcommand{\xx}{\boldsymbol{x}}

\newcommand{\mcP}{\mathcal P}
\newcommand{\Ni}{\mathcal N^{(i)}}

\newcommand{\T}{\mathcal T}
\newcommand{\ZZ}{\mathbb Z}
\newcommand{\QQ}{\mathbb Q}
\newcommand{\QT}{\mathcal{T}}
\newcommand{\VV}{V_{\lambda}^*}
\newcommand{\G}{{Q^{\pm}}}
\newcommand{\mcG}{\mathcal{G}}
\newcommand{\DD}{D} 
\newcommand{\NN}{\mathbb N}
\newcommand{\YY}{\mathbb Y}
\newcommand{\SG}{S}
\newcommand{\hatfstar}[1]{\widehat{f^*_{#1}}}

\usepackage[colorlinks,cite color=green,link color=purple]{hyperref}
\definecolor{green}{RGB}{43,92,47}
\definecolor{blue}{RGB}{40,68,104}
\definecolor{red}{RGB}{254, 113, 96}
\definecolor{purple}{RGB}{102,0,51}
\definecolor{gray}{RGB}{224,224,224}
\definecolor{lightpurple}{RGB}{255, 249, 242}
\definecolor{blue}{RGB}{40,68,104}
\hypersetup{colorlinks = true, urlcolor=blue}

\makeatletter
\newcommand\footnoteref[1]{\protected@xdef\@thefnmark{\ref{#1}}\@footnotemark}
\makeatother

\title{A combinatorial formula for Interpolation Macdonald polynomials}
\date{\today}
\author{Houcine Ben Dali}
\address{Department of Mathematics,
Harvard University, Cambridge, MA, and 
Center for Mathematical Sciences and Applications, Harvard University, Cambridge, MA}
\email{bendali@math.harvard.edu}
\author{Lauren Kiyomi Williams}
\address{Department of Mathematics,
Harvard University, Cambridge, MA}
\email{williams@math.harvard.edu}

\begin{document}

\begin{abstract}
In 1996, Knop and Sahi introduced a remarkable
family of inhomogeneous symmetric polynomials,  defined via vanishing conditions, whose top homogeneous parts are exactly the \emph{Macdonald polynomials}.  Like the Macdonald polynomials, 
these \emph{interpolation Macdonald polynomials} are closely connected to the Hecke algebra, and admit nonsymmetric versions, which generalize the nonsymmetric Macdonald polynomials.  In this paper we give a combinatorial formula for interpolation Macdonald polynomials in terms of \emph{signed multiline queues}; 
this formula generalizes the combinatorial formula for Macdonald polynomials in terms of multiline queues given by Corteel--Mandelshtam--Williams.
\end{abstract}

\maketitle

\begin{center}
\begin{minipage}{0.33\textwidth}
 \begin{center}
  \emph{Interpolation\\ Macdonald polynomials?\\ E-I-E-I-O!}
     \end{center}
\end{minipage}
\end{center}

\setcounter{tocdepth}{1}
\tableofcontents

\section{Introduction}

\subsection{Interpolation polynomials}
\emph{Macdonald polynomials}, introduced by Ian Macdonald in 1989 \cite{Macdonald1988}, are one of the most interesting families of polynomials in mathematics: they have connections to the geometry of the Hilbert scheme \cite{Haiman2001}, and admit various
beautiful combinatorial formulas in terms of tableaux~\cite{HaglundHaimanLoehr2005}, \emph{multiline queues}~\cite{CorteelMandelshtamWilliams2022}, and 
\emph{vertex models}~\cite{ABW}.
There is a fascinating inhomogeneous generalization  of Macdonald polynomials
called \emph{interpolation Macdonald polynomials}, introduced by Knop \cite{Knop1997b}
and Sahi \cite{Sahi1996} around 1996, and further studied in \cite{Okounkov1998,Olshanski2019}.
These polynomials are related to gauge theories and vertex operators \cite{CNO},  the HOMFLY polynomial \cite{Kameyama}, and the theory of link invariants of $\mathfrak{gl}_n$~\cite{BeliakovaGorsky2024}. In the Jack limit, interpolation polynomials were recently proved to be monomial-positive~\cite{NakviSahiSergel2023} and shown to be closely related to the theory of non-orientable combinatorial maps ~\cite{BenDaliDolega2023}.

The main result of this paper is a combinatorial formula for interpolation Macdonald polynomials. These polynomials can be  defined via vanishing conditions as in \cref{def:intMacdonald}.

Given a composition $\mu=(\mu_1,\dots,\mu_n)\in\NN^n$, we define
\begin{align} \label{eq:k}
k_i(\mu)&:=\#\{j:j<i \text{ and }\mu_j>\mu_i\}+\#\{j:j>i \text{ and }\mu_j\geq \mu_i\}, \text{ and }\\\label{eq:k2}
\widetilde\mu&:=\left(q^{\mu_1} t^{-k_1(\mu)},\dots,q^{\mu_n} t^{-k_n(\mu)}\right).
\end{align}
For example,  when $\mu=(4,2,0,1,4)$ we have
    $\widetilde \mu=(q^4t^{-1},q^2t^{-2},t^{-4},qt^{-3},q^4)$.

\begin{thm}\cite{Knop1997b, Sahi1996}\label{def:intMacdonald}
For each partition $\lambda=(\lambda_1,\dots,\lambda_n)$, there is a unique inhomogeneous symmetric polynomial $P^*_{\lambda} =P_{\lambda}^*(\xx;q,t) = P_{\lambda}^*(x_1,\dots,x_n;q,t)$ 
called the \emph{interpolation Macdonald polynomial} such that 
\begin{itemize}
    \item the coefficient $[m_{\lambda}]P_{\lambda}^*$ of the 
    monomial symmetric polynomial $m_{\lambda}$ in $P_{\lambda}^*$ is $1$,
    \item $P_{\lambda}^*(\tilde{\nu}) = 0$ for each partition $\nu \neq \lambda$ with $|\nu| \leq |\lambda|$.
\end{itemize}
Moreover, the top homogeneous component of 
$P_{\lambda}^*$ is the usual Macdonald polynomial
$P_{\lambda}.$
\end{thm}

Recall that there are also \emph{nonsymmetric Macdonald polynomials} $E_{\mu}$, introduced by 
Cherednik \cite{cherednik-1995b}, associated to any composition
$\mu\in \NN^n$; these also have
interpolation analogues $E_{\mu}^*$ due to Knop and Sahi, see \cref{thm:Knop_Sahi}.  More recently the so-called 
\emph{ASEP polynomials} $f_{\mu}$ 
were introduced in connection to 
the \emph{asymmetric simple exclusion process} (ASEP), see 
\cite{cantini-degier-wheeler-2015, chen-degier-wheeler-2020}.
The ASEP polynomials are in fact
special cases
of the
\emph{permuted-basement Macdonald polynomials}
introduced in
\cite{ferreira-2011}, as shown in
\cite{CorteelMandelshtamWilliams2022}.

In this article we define 
\emph{interpolation ASEP polynomials} as in \cref{def:ASEP} below; they have the property that their top homogeneous component recovers the usual ASEP polynomials.  Our main result is a combinatorial formula for both the interpolation ASEP polynomials and the interpolation Macdonald polynomials, see \cref{thm:main}.

\begin{definition}\label{def:ASEP}
Fix a partition $\lambda$.
    For $\mu\in\SG_n(\lambda)$, the \emph{ASEP polynomial} $f_\mu$ is the homogeneous polynomial defined by
    $$f_\mu=T_{\sigma_\mu}\cdot E_\lambda,$$
    where 
$\sigma_{\mu}$ is the shortest permutation in $S_n$ such that 
$\sigma_{\mu}(\lambda)=\mu$, see \eqref{eq:action} and \eqref{def:Hecke} for the notation. In particular, $f_{\lambda} = E_{\lambda}.$
    
Similarly, we define the \emph{interpolation ASEP polynomial} $f^*$ by 
$$f^*_\mu:=T_{\sigma_\mu}\cdot E^*_\lambda.$$
In particular, $f^*_{\lambda} = E^*_{\lambda}.$
\end{definition}
Since the top homogeneous part of $E^*_\lambda$ is $E_\lambda$, we get that the top homogeneous part of $f^*_\mu$ is then the ASEP polynomial $f_\mu$.  In particular, the degree of $f^*_{\mu}$ is $|\mu|$.  In \cref{ssec:interpolation_ASEP}, we give a characterization of interpolation ASEP polynomials with vanishing conditions.

\subsection{Multiline queues and signed multiline queues}

Let $\lambda=(\lambda_1,\dots, \lambda_n)$
with $\lambda_1\geq\dots\geq\lambda_n\geq 0$ be a partition. We can
describe such a partition by its vector of types
$\mathbf{m}=(m_0, m_1, \dots, m_L)$, where
$m_i=\#\{j:\lambda_j=i\}$,
and $L$ is the largest part that occurs.
Sometimes we denote our partition by
$\lambda=\langle L^{m_L}, \dots, 1^{m_1}, 0^{m_0} \rangle$.
We have $\sum_{i=0}^L m_i=n$.

\begin{definition}\label{def:MLQ0}
Fix a partition
$\lambda=\langle L^{m_L}, \dots, 1^{m_1}, 0^{m_0} \rangle$
        as above, with $\sum_{i=0}^L m_i=n$.
A \emph{ball system} $B$ of type $\lambda$
is an $L \times n$ array, with rows labeled from bottom to top as $1, 2, \dots, L$,
and columns labeled from left to right from $1$ to $n$,
in which each of the $Ln$ positions is either empty or occupied by a ball, and such that
there are
        $m_L+ m_{L-1} + \dots + m_r$ balls in row $r$.
        We label each ball with an element of $\{1, \dots, L\}$ (viewing
        empty spots as $0$), such that:
        \begin{itemize}
                \item in row $r$ our configuration of balls gives a permutation of
                        $$\lambda^{(r)}:=\langle L^{m_L}, (L-1)^{m_{L-1}}, \dots, r^{m_r},  0^{m_{r-1} + \dots + m_0} \rangle.$$
        \end{itemize}
\end{definition}
\cref{def:MLQ} comes from \cite{CorteelMandelshtamWilliams2022}, and is a slight variant of a definition from \cite{Martin}.
\begin{definition}\label{def:MLQ}
A \emph{multiline queue} (or MLQ) of \emph{type} 
$\mu\in S_n(\lambda)$ is a ball system of type $\lambda$ such that 
each ball  in row $r>1$ is paired with a ball  of the same label
in the row below it, and the configuration of balls on the bottom row is $\mu$.  We require that the set of pairings between row $r$ and $r-1$ form a \emph{classic layer}, i.e. satisfy the following rules:
\begin{itemize}
 \item  We pair two balls using a shortest strand
that travels either straight down or from left to right, allowing the
strand to wrap around the cylinder if necessary;
\item  In row $r$, each  ball with label $a$ has either an empty spot below it,
           or a ball with label $a'$, where $a' \geq a$, and if $a=a'$, they must be trivially paired,
                i.e. paired to each other with a straight segment.
\end{itemize}
Let $\MLQ(\mu)$ denote the set of multiline queues of type $\mu$.
\end{definition}
See \Cref{fig:MLQ_example} for an example.
\begin{figure}[!ht]
  \centerline{\includegraphics[height=1.2in]{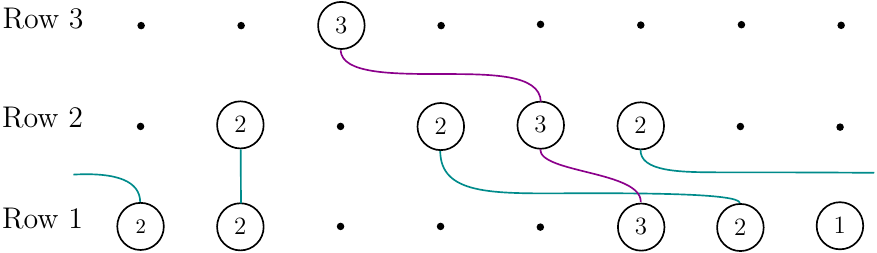}}
\centering
 \caption{
A multiline queue of type $(2,2,0,0,0,3,2,1)$.}
\label{fig:MLQ_example}
 \end{figure}

        \begin{definition}
An \emph{enhanced ball system} $B$ of type $\lambda$
is a $2L \times n$ array, with rows labeled from bottom to top as $1, 1', 2, 2', \dots, L,L'$,
and columns labeled from left to right from $1$ to $n$,
in which each of the $2Ln$ positions is either empty or occupied by a ball, and such that
there are
        $m_L+ m_{L-1} + \dots + m_r$ balls in each of rows $r$ and $r'$.
        Moreover:
        \begin{itemize}
                \item[(a)] in row $r$ our balls are labeled by $\{1,2,\dots,L\}$ (we call them \emph{regular balls}) and the configuration of balls gives a permutation of
                        $$\lambda^{(r)}:=\langle L^{m_L}, (L-1)^{m_{L-1}}, \dots, r^{m_r},  0^{m_{r-1} + \dots + m_0} \rangle$$
                \item[(b)] in row $r'$ our balls are labeled by $\{\pm 1, \dots, \pm L\}$ (we call them \emph{signed balls}) and the configuration of balls gives a \emph{signed} permutation of
                        $$\lambda^{(r)}=\langle L^{m_L}, (L-1)^{m_{L-1}}, \dots, r^{m_r},  0^{m_{r-1} + \dots + m_0} \rangle$$
        \end{itemize}
        A signed ball with a positive (respectively negative) label will be called \emph{a positive ball} (respectively \emph{a negative ball}). 
\end{definition}

\begin{definition}  \label{def:ghostMLQ}
A \emph{signed multiline queue} $\G$ of type $\mu\in S_n(\lambda)$ is an enhanced ball system of type $\lambda$
such that each ball in a row above row $1$ is paired with
a ball  of
the same absolute value in the row below it,
and the configuration of balls on the bottom row is $\mu$.
We require that, if we consider only the absolute values of the ball labels, then 
the pairings between row $r$ and row $(r-1)'$ form 
a classic layer, as in \cref{def:MLQ}, and we call them
\emph{classic pairings}.
And we require that the pairings between row $r'$ and row $r$ form a \emph{signed layer}, i.e. satisfy the following rules (and we call them \emph{signed pairings}):
\begin{itemize}
    \item[(a')] Each pairing connects two balls with a shortest strand that travels either straight down or from left to right, and does not wrap around;
    \item[(b')] In row $r'$, each positive ball with label $a\in \ZZ^+$
            must always have a ball labeled $a'$ underneath it, where $a' \geq a$, and if $a'=a$, the two balls must be trivially paired;
    \item[(c')] In row $r'$, each negative ball with label $-a$ (for $a\in \ZZ^+$) has either an empty spot below it or a ball with label $a'$, where $a \geq a'$.
           \end{itemize}
Let $\SMLQ(\mu)$ denote the set of signed multiline queues of type $\mu$.
\end{definition}

\begin{figure}[t]
    \centering
    \begin{subfigure}{0.3\textwidth}
        \centering
        \includegraphics[height=2.8cm]{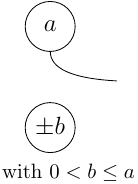}
        \subcaption{}
        \label{fig:forbidden_configurations_classic}
    \end{subfigure}
    \begin{subfigure}{0.6\textwidth}
        \centering
        \includegraphics[height=2.8cm]{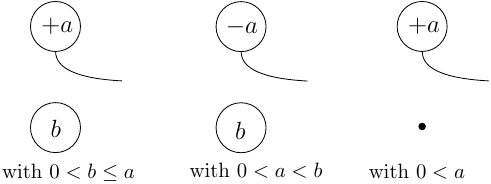}
        \subcaption{}
        \label{fig:forbidden_configurations_ghost}
    \end{subfigure}
    \caption{\Cref{fig:forbidden_configurations_classic} illustrates the forbidden configurations for the classic layers, while the three other figures (\Cref{fig:forbidden_configurations_ghost}) show the forbidden configurations for the signed layers. The three figures on the left show two balls on top of each other, which are not trivially paired, whereas the rightmost figure features a regular ball with an empty position beneath it.}
    \label{fig:forbidden_configurations}
\end{figure}
In \Cref{fig:forbidden_configurations_classic} and \Cref{fig:forbidden_configurations_ghost}
we illustrate the forbidden configurations in the classic and signed layers, respectively.

See \Cref{fig:GMLQ_example} for an example of a signed multiline queue.

\begin{figure}[!ht]
  \centerline{\includegraphics[height=2in]{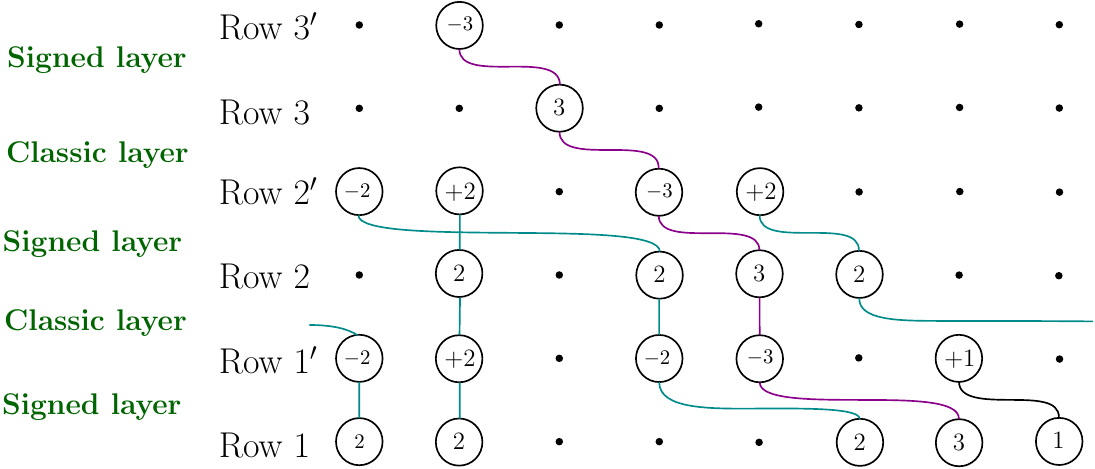}}
\centering
 \caption{
A signed multiline queue of type $(2,2,0,0,0,2,3,1)$.
}
\label{fig:GMLQ_example}
 \end{figure}

\begin{remark}
In \cite{Martin}, multiline queues were given an interpretation in terms of ``priority queues,'' with balls at each level representing customers who are each seeking a service on the level below. 
We can also interpret our signed multiline queues as follows.  The signed balls all represent customers, with the positive balls having ``polite'' and ``attractive'' characteristics and the negative balls having ``needy'' characteristics.  The regular balls all represent services; they are always ``polite''.  We will work our way from top to bottom of the multiline queue: in Row $r'$ the customers seek a service from Row $r$ below, and in Row $r$ the services seek a customer from Row $(r-1)'$ below. When we construct pairings between two adjacent rows, we will start by examining the balls of largest absolute value in the higher row; if we need to break ties, we will work from right to left. The polite customers and services can choose any unused service/customer below, except that if there is an unused service immediately underneath them, politeness dictates that they must accept it; this explains the leftmost two diagrams in 
\Cref{fig:forbidden_configurations}.  When we construct the pairings on a signed layer, because the negative balls/customers are so needy, 
no ball (positive or negative) dares to pair with an unused service that is immediately below a negative customer who has not yet accepted a service; 
this explains the third diagram in \Cref{fig:forbidden_configurations}. Since the positive balls/customers are attractive, there is always a service to be found just underneath them (though it may be taken already); this explains the fourth diagram in 
\Cref{fig:forbidden_configurations}.  Finally, pairings initiated by services from Row $r$ can wrap around, because servers ``know the building''; however, pairings initiated by customers cannot.  
\end{remark}

\subsection{Combinatorial formulas for ASEP and interpolation ASEP polynomials}

In this section we define weights for both multiline and signed multiline queues.  We then use them to review the combinatorial formula for ASEP polynomials and give a new combinatorial formula for interpolation ASEP polynomials.

\begin{definition}\label{def:wt0}
Let $Q$ be a multiline queue.  If the balls in row $r$ form the composition $\mu=(\mu_1,\dots,\mu_n)$,
we define the \emph{ball-weight} of row $r$ and of $Q$ to be 
\begin{equation}\label{eq:ball0}
\wt_{\ball}(r) = \prod_{i: \mu_i>0} x_i 
\hspace{.5cm} \text{ and } \hspace{.5cm}
\wt_{\ball}(Q) = \prod_{r=1}^L \wt_{\ball}(r).
\end{equation}

We also define the \emph{pairing-weight} $\wt_{\pair}(Q)$ of $Q$ by associating a weight to each nontrivial pairing $p$ of balls. 
Consider the pairings  
in a (necessarily) classic layer
connecting balls in row $r$ and row $(r-1)$.  Their weights are computed via the following \emph{pairing order}. 
We read the balls in row $r$ in decreasing order of their label; 
within a fixed label, we read the balls from right to left.
As we read the balls in this order, we imagine placing the strands pairing the balls one by one.  The balls in row $(r-1)$ that have not yet been matched right before we place $p$ are considered  \emph{free} for $p$. If pairing $p$ matches a ball labeled $a$ in row $r$ and column $j$ to a ball
in row $(r-1)$ and column $j'$, then the free balls in row $(r-1)$ and columns $j+1, j+2,\dots, j'-1$  (indices considered modulo $n$) are considered \emph{skipped}. (When pairing balls of label $a$ between rows $r$ and $(r-1)$, trivially paired balls of label $a$ in row $(r-1)$ are not considered free.)  We then associate to pairing $p$ the weight
\begin{equation}\label{eq:pair}
\wt_{\pair}(p) = \begin{cases}
\frac{(1-t) t^{\skipped(p)}}{1-q^{a-r+1}t^{\free(p)}}\cdot q^{a-r+1}
&\mbox {if $j'<j$}\\
\frac{(1-t) t^{\skipped(p)}}{1-q^{a-r+1}t^{\free(p)}}
&\mbox{if $j'>j$}.
\end{cases}
\end{equation}
Note that the  factor $q^{a-r+1}$ appears precisely when
the pairing wraps around the cylinder.

Having associated a weight to each nontrivial pairing,
we define $$\wt_{\pair}(Q) = \prod_p \wt_{\pair}(p),$$
where the product is over all nontrivial pairings of balls in $Q$.

Finally the \emph{weight} of $Q$ is defined to be
$$\wt(Q) = \wt_{\ball}(Q) \wt_{\pair}(Q).$$
\end{definition}

\begin{definition}\label{def:Fmu0}
Let $\mu = (\mu_1,\dots, \mu_n) \in \{0,1,\dots, L\}^n$ be a composition with largest part $L$.
We set
\[
F_{\mu} = F_{\mu}(x_1,\dots,x_n; q, t)  = F_{\mu}({\xx}; q, t) 
 = \sum_{Q\in \MLQ(\mu)} \wt(Q).\]
 
Let $\lambda = (\lambda_1 \geq \lambda_2 \geq \dots \geq \lambda_n \geq 0)$
be a partition with $n$ parts and largest part $L$.
We set
$$Z_{\lambda} = Z_{\lambda}(x_1,\dots,x_n; q, t) = 
        Z_{\lambda}({\xx}; q, t) = \sum_{\mu\in S_n(\lambda)} F_{\mu}(x_1,\dots,x_n; q, t).$$
        We call
$Z_{\lambda}$ the \emph{combinatorial partition function} for multiline queues.
\end{definition}

\begin{thm}
\label{thm:main0}\cite{CorteelMandelshtamWilliams2022}
Let $\mu\in \NN^n$ be a  composition, and let $\lambda$ be a partition.  Then the ASEP polynomial $f_{\mu}$ equals the 
weight-generating function $F_{\mu}$ for multiline queues of type $\mu$. And the Macdonald polynomial $P_{\lambda}(\xx; q, t)$
        is equal to
 the combinatorial partition function $Z_{\lambda}(\xx; q, t)$ for  multiline queues.
\end{thm}

Our goal is now to give an analogue
of \cref{thm:main0} for interpolation
 polynomials.

\begin{definition}\label{def:wt}
Let $\G$ be a signed multiline queues.  If the balls in row $r'$ form the signed composition $\alpha=(\alpha_1,\dots,\alpha_n)$,
we define the \emph{shifted ball-weight} of row $r'$ to be 
\begin{equation}\label{eq:ball}
\wt_{\ball}(r') = \left(\prod_{i: \alpha_i>0} x_i \right) \left(\prod_{i: \alpha_i<0} \frac{-q^{r-1}}{t^{n-1}} \right)
\end{equation}
and we define the \emph{shifted ball-weight} of $\G$ to be 
\begin{equation}
    \wt_{\ball}(\G) = \prod_{r=1}^L \wt_{\ball}(r').
\end{equation}
In other terms, we assign to a  ball in column $i$ and row $r'$ the weight $x_i$ if it is a positive ball and the weight $\frac{-q^{r-1}}{t^{n-1}}$ if it is a negative ball.

We also define the \emph{pairing-weight} $\wt_{\pair}(\G)$ of $\G$ by associating a weight to each nontrivial pairing $p$ of balls. 
For the pairings in a classic layer
connecting balls in row $r$ and row $(r-1)'$, we use the weighting scheme given in \eqref{eq:pair}, where we ignore the signs on ball labels and only work with the absolute value.

For the pairings in a signed layer connecting balls in row $r'$ and row $r$, we read the balls in row $r'$ in decreasing order of the absolute value of their label; within a fixed absolute value, we read the balls from right to left.  Reading the balls in this order, we place the strands pairing the balls one by one.  The balls in row $r$ that have not yet been matched  are \emph{free}.  If pairing $p$ matches a ball labeled $\pm a$ in row $r'$ and column $j$ to a ball labeled $a$ in row $r$ and column $k>j$, then the free balls (respectively, empty positions) in row $r$ and columns $j+1,j+2,\dots,k-1$
are  \emph{skipped} (respectively, \emph{empty}).  We then set 
\begin{equation}\label{eq:pair2}
\wt_{\pair}(p) = \begin{cases}
         (1-t) t^{\skipped(p)+\emp(p)} &\text{ if $p$ connects a positive ball and a regular ball}\\
         -(1-t) t^{\skipped(p)+\emp(p)} &\text{ if $p$ connects a negative ball and a regular ball.}
\end{cases}
\end{equation}

Having associated a weight to each nontrivial pairing,
we define $$\wt_{\pair}(\G) = \prod_p \wt_{\pair}(p),$$
where the product is over all nontrivial pairings of balls in $\G$.

Finally the \emph{weight} of $\G$ is defined to be
$$\wt(\G) = \wt_{\ball}(\G) \wt_{\pair}(\G).$$
\end{definition}

\begin{remark}
    If all the balls in our signed multiline queue are regular, i.e all labels are in $\NN_+$, then it follows from items (a') and (b') of \cref{def:ghostMLQ} that all the ghost pairings are trivial. As a consequence, the contribution of the signed layers to the pairing weight of the system is 1. We can then remove these layers and keep only the classic ones; the definition of signed multiline queue then reduces to \cref{def:MLQ}.
\end{remark}

\begin{example}
In \Cref{fig:GMLQ_example}, the ball-weight of $\G$ is 
$$\left(\frac{-q^2}{t^7}\right) \cdot x_2 x_5 \left(\frac{-q}{t^7}\right)^2
\cdot x_2 x_7 \left(\frac{-1}{t^7}\right)^3.$$
Meanwhile, the weights of the nontrivial pairings are as follows (reading from left to right):
\begin{itemize}
    \item From Row $3'$ to Row $3$: $-(1-t)$
    \item From Row $3$ to Row $2'$: $\frac{1-t}{1-qt^4}$
    \item From Row $2'$ to Row $2$: $-t(1-t)$, $-(1-t)$, and $(1-t)$
    \item From Row $2$ to Row $1'$: $\frac{(1-t)t}{1-qt^2} \cdot q$
    \item From Row $1'$ to Row $1$: $-t(1-t)$, $-t(1-t)$, and $(1-t)$.
        \end{itemize}
Thus, multiplying all of these weights, we obtain 
$$\wt(\G) = -x_2^2 x_5 x_7 \frac{q^5(1-t)^9}{t^{38}(1-qt^2)(1-qt^4)}.$$
\end{example}

Notice that in signed multiline queues, the pairing weights do not depend on the signs of the labels, only on their absolute value. However, the signs play an important role in the forbidden configurations and the ball weights.

We now define the weight-generating function for signed multiline queues
of a given type, as well as the \emph{combinatorial partition function}
for signed multiline queues.
\begin{definition}\label{def:Fmu}
Let $\mu = (\mu_1,\dots, \mu_n) \in \{0,1,\dots, L\}^n$ be a composition with largest part $L$.
We set
\[
F^*_{\mu} = F^*_{\mu}(x_1,\dots,x_n; q, t)  = F^*_{\mu}({\xx}; q, t) 
 = \sum_{\G \in \SMLQ(\mu)} \wt(\G).\]

Let $\lambda = (\lambda_1 \geq \lambda_2 \geq \dots \geq \lambda_n \geq 0)$
be a partition with $n$ parts and largest part $L$.
We set
$$Z^*_{\lambda} = Z^*_{\lambda}(x_1,\dots,x_n; q, t) = 
        Z^*_{\lambda}({\xx}; q, t) = \sum_{\mu\in S_n(\lambda)} F^*_{\mu}(x_1,\dots,x_n; q, t).$$
        We call
$Z^*_{\lambda}$ the \emph{combinatorial partition function} for signed multiline queues.
\end{definition}

\begin{thm}[Main theorem]
\label{thm:main}
Let $\mu$ be a  composition, and let $\lambda$ be a partition.  Then the interpolation ASEP polynomial $f_{\mu}^*$ equals the 
weight-generating function $F_{\mu}^*$ for signed multiline queues of type $\mu$. And the interpolation Macdonald polynomial $P^*_{\lambda}(\xx; q, t)$
        is equal to
 the combinatorial partition function $Z^*_{\lambda}(\xx; q, t)$ for signed multiline queues.
\end{thm}

\begin{example}
To use \cref{thm:main} to compute the interpolation ASEP polynomial $f^*_{(0,2)}$, we enumerate all signed multiline queues of type $(0,2)$, see \Cref{fig:GMLQ_example2}, and then sum up 
their weights, obtaining
\begin{align*}
f^*_{(0,2)}=&\frac{1-t}{1-qt}(x_1-q/t)(x_2-1/t) + \frac{1-t}{t}(x_1-q/t) + (x_2-q/t)(x_2-1/t) + \\
&(1-t) \frac{q}{t} (x_2-1/t) + \frac{q^2}{t^2} \frac{(1-t)^3}{1-qt} + \frac{q}{t} \frac{(1-t)^2}{1-qt} (x_2-q/t).
\end{align*}
The usual ASEP polynomial
is the top homogeneous part of the expression above, namely 
$$f_{(0,2)}=\frac{1-t}{1-qt} x_1 x_2 + x_2^2.$$
This can be computed from the 
signed multiline queues which have no negative balls, and whose pairings from row $r'$ to row $r$ are all trivial.

    \begin{figure}[!ht]
  \centerline{\includegraphics[height=2.8in]{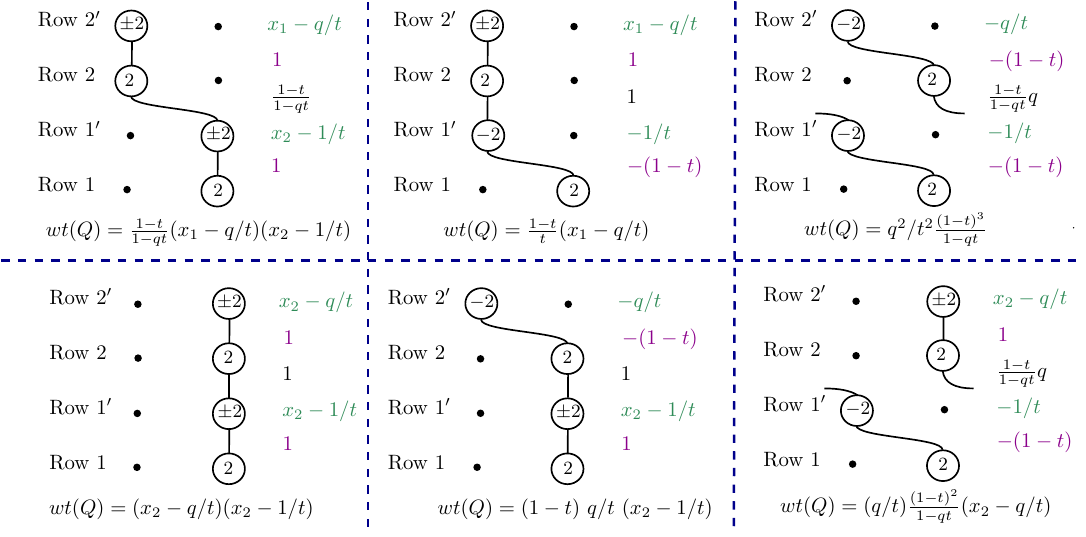}}
\centering
 \caption{The signed multiline queues of type $(0,2)$, with 
 their weights superimposed.  Note that 
 a ball labeled $\pm 2$ represents
 the fact that the corresponding ball can 
 either be a positive or a negative ball.  Thus, the six diagrams above actually represent $15$ signed multiline queues.}
\label{fig:GMLQ_example2}
 \end{figure}
\end{example}

\begin{remark}
While our combinatorial formulas for interpolation ASEP and Macdonald polynomials can be seen as generalizations of the combinatorial formulas in 
\cite{CorteelMandelshtamWilliams2022}, our proof strategy is quite different.  The proof in \cite{CorteelMandelshtamWilliams2022} utilized the
circular symmetry
$$q^{\mu_n} f_{\mu}(\mathbf{x}; q, t)=
f_{\mu_n, \mu_1,\dots,\mu_{n-1}}(qx_n, x_1,\dots, x_{n-1}; q, t)$$ for ASEP polynomials; however,  interpolation ASEP polynomials lack this property.
\end{remark}

When $q=1$, the ASEP polynomials and the Macdonald polynomials have a probabilistic interpretation in terms of the $t$-Push TASEP \cite{AMW}.  We will give an interpolation analogue of this result in \cite{BenDaliWilliams2}, using the recursive structure of how signed multiline queues are built.  

\subsection{Integrality and Comparison with Okounkov's Formula}

Knop and Sahi proved that the \emph{integral form} of the interpolation  Macdonald polynomials $P_\lambda^*$ satisfy an integrality property (see e.g \cite[Corollary 5.5]{Knop1997b}). This can be proved using our combinatorial formula for $P_\lambda^*$, and extended to interpolation ASEP polynomials $f^*_\mu$; see \cref{ssec:integral_form}.

In \cite{Okounkov1998}, Okounkov gave a combinatorial formula for the interpolation symmetric Macdonald polynomials, which, to our knowledge, was the only such formula prior to our work. This formula is obtained by ``shifting'' Macdonald's formula for the homogeneous symmetric Macdonald polynomials~\cite[Section VI.7]{Macdonald1995}. It expresses the polynomial $P_\lambda^*$ as a sum over tableaux of shape $\lambda$, counted with coefficients given by products of \emph{Pieri coefficients}. These coefficients are quite complicated to compute, and in particular, the integrality property is not apparent from this tableau formula.

\bigskip

The structure of this paper is as follows.
In \cref{sec:background} we provide background on 
interpolation polynomials; we also define interpolation ASEP polynomials, and give vanishing condition characterization of them. In \cref{sec:packed}
we provide a recursion for interpolation ASEP polynomials from \emph{packed} compositions; this
provides a base case for our subsequent arguments.  We generalize this recursion to arbitrary compositions in \cref{sec:tworow}.  In \cref{sec:proof} we provide a combinatorial analysis of two-line signed multiline queues, and complete the proof of the main theorem. 
In \cref{sec:tableau} we give a tableaux formula for interpolation ASEP and Macdonald polynomials and prove an integrality result for them.
Finally in 
\cref{sec:factorization} we give a factorization property for interpolation Macdonald polynomials at $q=1$.

\bigskip
\noindent{\bf Acknowledgements:~}
We would like to thank Olya Mandelshtam for several very useful discussions.
HBD acknowledges support from the Center of Mathematical Sciences and Applications at Harvard University.
LW was supported by the National Science Foundation under Award No.
DMS-2152991 until May 12, 2025, when the grant was terminated; she would also like to thank the Radcliffe Institute for Advanced Study, where some of this work was carried out. 
Any opinions, findings, and conclusions or recommendations expressed in this material are
those of the author(s) and do not necessarily reflect the views of the National Science
Foundation.

\section{Background on interpolation polynomials}\label{sec:background}

We now provide some more background on interpolation polynomials.
We also prove some properties of interpolation ASEP polynomials.

\subsection{Notation}\label{ssec:notation}
Fix $n\geq 1$. Let $\YY_n$ denote the set of integer partitions 
$\lambda=(\lambda_1,\dots,\lambda_n)=(\lambda_1 \geq \dots \geq \lambda_n)$ with at most $n$ parts.  We let $|\lambda|$ denote the sum $\lambda_1+ \dots + \lambda_n$ of the parts of the partition and call it the \emph{size} of $\lambda$.

Let $\mcP_n$ denote the ring of polynomials in $n$ variables, and 
let $\mcP_n^{(d)}$ denote the polynomials of degree at most $d$. Similarly, let $\Lambda_n$ denote the ring of symmetric polynomials with $n$ variables and let $\Lambda_n^{(d)}$ denote the symmetric polynomials with degree at most $d$.
All the polynomials considered here will have coefficients in $\QQ(q,t).$

The symmetric group acts on $\NN^n$ by
\begin{equation} \label{eq:action}
\sigma\cdot (\mu_1,\mu_2,\dots,\mu_n):=(\mu_{\sigma^{-1}(1)},\mu_{\sigma^{-1}(2)},\dots,\mu_{\sigma^{-1}(n)}).
\end{equation}
For $\mu\in\NN^n$ we will write
$x^\mu:=x_1^{\mu_1}\dots x_n^{\mu_n}$. The symmetric group acts on $\mcP_n$ by
$$\sigma(x^\mu):=x_{\sigma(1)}^{\mu_1}\dots x_{\sigma(N)}^{\mu_n}=x^{\sigma(\mu)}.$$

\subsection{Interpolation Macdonald polynomials}
We now recall some of the main results of \cite{Knop1997b,Sahi1996}.
Recall the notation $\widetilde{\mu}$ from \eqref{eq:k2}. 
\begin{thm}[{\cite[Theorem 2.2]{Knop1997b}}]\label{thm:Knop-Sahi-interpolation}
    Fix two integers $d,n\geq 1$, and fix a family $(a_\nu)_{\nu\in\NN^n,|\nu|\leq d}$ in $\QQ(q,t)$. Then there exists a unique polynomial $f\in\mcP_n^{(d)}$ such that for any $|\nu|\leq d$ we have $f(\widetilde{ \nu})=a_\nu$.

    In particular, if $f\in\mcP_n^{(d)}$ such that for any $|\nu|\leq d$ we have $f(\widetilde{\nu})=0$, then $f=0$.
\end{thm}

\begin{remark}
 Our notation is similar to but not identical to that of \cite{Knop1997b}.
In particular, we have
$$(\widetilde \mu)^{\rev}=\widebar{\mu^{\rev}},$$
where the sequence $\widebar \mu$ is the one from \cite{Knop1997b}, and $\nu^{\rev}:=(\nu_n,\dots,\nu_1)$.  When $f$ is symmetric, we have $f(\widetilde \nu)=f(\widebar \nu)$.
\end{remark}

Polynomials defined by their evaluation on compositions as in \cref{thm:Knop-Sahi-interpolation} are known as \emph{interpolation polynomials}. 
\begin{thm}[\cite{Knop1997b,Sahi1996}]\label{thm:Knop_Sahi}
    Fix $\mu\in \NN^n$ of size $d$. There exists a unique polynomial $E^*_\mu\in\mcP_n^{(d)}$, called the \emph{nonsymmetric interpolation Macdonald polynomial}, such that
    \begin{itemize}
        \item $[x^\mu]E^*_\mu=1$ (so in particular, $E^*_{\mu}$ has degree $d$),
        \item $E^*_\mu(\widetilde\nu)=0$ for any $\nu\in \NN^n$ satisfying $|\nu|\leq d$ and $\nu\neq \mu$.
    \end{itemize}
    Moreover, the top homogeneous part of $E^*_\mu$ is the nonsymmetric Macdonald polynomial $E_\mu$.
\end{thm}
Note that the first part of \cref{thm:Knop_Sahi} is a consequence of \cref{thm:Knop-Sahi-interpolation}. The second part giving the connection to Macdonald polynomials is however more surprising.
\begin{remark}\label{rmk:Enu(nu)}
    For any $\nu\in\NN^n$, we have $E_\nu^*(\widetilde{\nu})\neq 0$. This is a consequence of \cref{thm:Knop-Sahi-interpolation} and the fact that $E_\nu$ is not identically zero by definition.
\end{remark}

For any partition $\lambda\in\YY_n$ of size $d$, we define the space $\VV\subset\mcP^{(d)}_n$ by
\\
$$\VV:=\left\{f\in\mcP_n^{(d)}|f(\widetilde\nu)=0 \text{ for any $|\nu|\leq |\lambda|$ and $\nu\notin\SG_n(\lambda)$}\right\}.$$
\begin{lem}\label{lem:span}
We have
$$\VV=\Span_{\QQ(q,t)} \left\{E^*_\mu|\mu\in\SG_n(\lambda)\right\}.$$
\end{lem}
\begin{proof}
The inclusion $\supseteq$ is direct from \cref{thm:Knop_Sahi}. We now fix $f$ in $\VV$ and we want to prove that $f$ is a linear combination of $E^*_\mu$ for $\mu\in\SG_n(\lambda)$. We define 
$$g(\xx)=\sum_{\nu\in\SG_n(\lambda)}\frac{f(\widetilde\mu)}{E^*_\nu(\widetilde{\nu})}E^*_\nu(\xx).$$
We claim that $f=g$. Indeed, $f(\widetilde \nu)=g(\widetilde \nu)$ for all compositions $|\nu|\leq |\lambda|$. Hence $f$ and $g$ are of degree at most $|\lambda|$ and agree on all compositions of size at most $|\lambda|$. By \cref{thm:Knop-Sahi-interpolation}, we get that $f=g$.    
\end{proof}

In a similar way, one shows that
    $\mcP_n=\oplus_{\lambda\in \YY_n} \VV$; see also \cite[Corollary 2.6]{Knop1997b}. In particular, $\{E^*_\mu:\mu\in \NN^n\}$ is a basis of $\mcP_n$.

\subsection{Hecke Operators}
For $1\leq i \leq n-1$, we let $s_i=(i,i+1)$ denote the transposition exchanging $i$ and $i+1$.
The \emph{Hecke operator} $T_i$, which acts on $\mcP_n$,
is defined by 
\begin{equation}\label{eq:Hecke}
T_i:=t-\frac{tx_i-x_{i+1}}{x_i-x_{i+1}}(1-s_i).
\end{equation}

These operators satisfy the relations of the Hecke algebra of type $A_{n-1}$
\begin{equation}\label{eq:Hecke_algebra}
\begin{array}{ll}
(T_i-t)(T_i+1)=0 & \text{for $1\leq i \leq n-1$}     \\
 T_iT_{i+1}T_i=T_{i+1}T_iT_{i+1}    & \text{for $1\leq i \leq n-2$}\\
 T_iT_j=T_jT_i    & \text{for $ |i-j| > 1$}.\\
\end{array}
\end{equation}

If $\sigma\in\SG_n$ and $\sigma=s_{i_1}\dots s_{i_\ell}$ is a reduced decomposition of $\sigma$, we define 
\begin{equation}\label{def:Hecke}
T_\sigma:=T_{i_1}\dots T_{i_\ell}.
\end{equation}
It follows from \eqref{eq:Hecke_algebra} that this definition is independent of the choice of reduced expression.

\begin{lem}
[{\cite[Corollary 3.2]{Knop1997b}}]\label{lem:stability}
For any $i$, we have that $T_i(\VV) \subset \VV.$  In particular, using \cref{lem:span}, 
we conclude that $T_i E_{\nu}^*$ is a linear combination of 
$E_{\mu}^*$ for $\mu\in S_n(\nu)$. 
\end{lem}

\begin{lem}[{\cite[Corollary 3.4]{Knop1997b}  or \cite[Theorem 4.5]{Sahi1996}}]\label{lem:symmetry_E_lambda}
    Let $\mu\in \NN^n$ such that $\mu_i=\mu_{i+1}.$
    Then $E^*_\mu$ is symmetric in $x_i$ and $x_{i+1}$, or equivalently
    $T_i E^*_\mu=t E^*_\mu.$  
\end{lem}

We note that Hecke operators can be used to define an inhomogeneous analogue of Cherednik operators, which act diagonally on non-symmetric Macdonald polynomials; see \cite{Knop1997b,Sahi1996}, or \cite[Section 2.6]{BDWv1}. 

\subsection{ASEP and interpolation ASEP polynomials}\label{ssec:ASEP}

Recall the definition of the
ASEP polynomial $f_\mu$ and the interpolation
ASEP polynomial $f^*_\mu$ from \cref{def:ASEP}.  They were defined in terms of the Hecke operators, and the permutation $\sigma_\mu\in\SG_n$, which is the shortest permutation such that  $\sigma_\mu(\lambda)=\mu$. Intuitively, $\sigma_\mu$ sends the left-most part of size $i$ in $\lambda$ to the left-most part of size $i$ in $\mu$, the second left-most to the second left-most, and so on.
\begin{example}
    If $\lambda=(4,4,3,3,1)$ and $\mu=(3,4,1,4,3)$, then $\sigma_\mu=(2,4,1,5,3)$.
\end{example}

\begin{lem}\label{lem:f_mu-V_lambda}
If $\mu\in \SG_n(\lambda)$, then $f^*_\mu\in \VV$.  
\end{lem}
\begin{proof}
    This follows from \cref{lem:stability}.
\end{proof}

\begin{prop}\label{prop:T_i-f_star}
    For $1 \leq i \leq n-1$, the interpolation ASEP polynomials $f^*_\mu$ satisfy:
    \begin{enumerate}
        \item\label{item 1} $T_i f^*_{\mu}=f^*_{s_i\mu}$ if $\mu_i>\mu_{i+1}$,
        \item\label{item 2} $T_if^*_{\mu}=tf^*_{\mu}$ if $\mu_i=\mu_{i+1}$,
        \item\label{item 3} $T_i f^*_{\mu}=(t-1)f^*_{\mu}+tf^*_{s_i\mu}$ if $\mu_i<\mu_{i+1}$.
    \end{enumerate}
\end{prop}
\begin{proof}
Let $\mu':=s_i\cdot \mu$. If $\mu_i>\mu_{i+1}$ then $\sigma_{\mu'}=s_i\sigma_\mu$.  Using the fact that $\ell(s_i\sigma_\mu)=\ell(\sigma_\mu)+1$, we get
$$f^*_{\mu'}=T_{\sigma_{\mu'}}\cdot E_\lambda^*=T_iT_{\sigma_\mu}\cdot E^*_\lambda=T_i\cdot f^*_\mu,$$
which gives \cref{item 1}.

We now assume that $\mu_i=\mu_{i+1}$. We then have  $\mu=\sigma_{\mu}(\lambda)=s_i\sigma_\mu(\lambda)$ which implies, by definition of $\sigma_\mu$, that $\ell(s_i\sigma_\mu)=\ell(\sigma_\mu)+1$. Hence, we get as above that $T_if^*_{\mu}=T_{s_i\sigma_\mu}\cdot E_\lambda^*.$

Consider now the transposition $s_j:=\sigma_\mu^{-1}s_i\sigma_\mu$. Since $s_i\mu=\mu$, we get $s_j \lambda=\lambda$. We deduce that $\ell(\sigma_\mu s_j)=\ell(s_i \sigma_\mu)=\ell(\sigma_\mu)+1.$ Hence, 
$$T_i\cdot f^*_{\mu}=T_{s_i\sigma_\mu} \cdot E_{\lambda}^*=T_{\sigma_\mu}T_j\cdot E_\lambda^*.$$
Using \cref{lem:symmetry_E_lambda} and the fact that $s_j\lambda=\lambda$ we deduce that 
\begin{equation*}
  T_i\cdot f^*_{\mu}=tT_{\sigma_\mu}\cdot E_\lambda^*=tf_\mu^*.\qedhere  
\end{equation*}
\cref{item 3} follows from \cref{item 1} and the relations of \cref{eq:Hecke_algebra}.
\end{proof}
\begin{remark}\label{rem:ASEPpoly}
It is well known that the usual ASEP polynomials also satisfy the relations of \cref{prop:T_i-f_star}, see e.g. \cite{CorteelMandelshtamWilliams2022}.  One can prove this using the same proof as above.
\end{remark}

\begin{lem}\label{lem:V-f}
Let $\lambda=(\lambda_1,\dots,\lambda_n)$ be a partition and let $V_{\lambda}:= \Span_{\QQ(q,t)}\{E_{\mu} \ \vert \ \mu\in S_n(\lambda)\}$.   The ASEP polynomials
$\{f_{\mu} \ \vert \ \mu\in S_n(\lambda)\}$ form a basis for the space $V_{\lambda}$.   
\end{lem}
\begin{proof}
The fact that $f_{\mu} \in V_{\lambda}$ 
comes from \cref{lem:span} and \cref{lem:f_mu-V_lambda} by taking the top homogeneous part.
Now using \cref{thm:main0}, it follows that 
for each $\tau\in S_n(\lambda)$, the only
$f_{\mu}$ for $\mu\in S_n(\lambda)$ which contains
the monomial $x^{\tau}$ with a nonzero coefficient
is $f_{\tau}$.  Thus, the elements of 
$\{f_{\mu} \ \vert \ \mu\in S_n(\lambda)\}$ are linearly independent, and hence form a basis.  See also \cite[Section 1]{CantinideGierWheeler2015} for a proof sketch of this result.
\end{proof}
\begin{cor} \label{lem:basis}
The polynomials $\{f_{\mu} \vert \ |\mu|=n\}$
form a basis for the space of polynomials
of degree $n$.    
\end{cor}

We have similar results for interpolation ASEP polynomials.
\begin{prop}\label{prop:V*-f*}
The interpolation ASEP polynomials
$\{f_{\mu}^* \ \vert \ \mu\in S_n(\lambda)\}$ form a basis for the space $V^*_{\lambda}$. As a consequence,  $\{f_{\mu}^* \ \vert \ \mu \in \NN^n\}$ is basis of $\mcP_n$. \end{prop}

\begin{proof}
    The fact that $f^*_\mu\in \VV$ was proven in \cref{lem:f_mu-V_lambda}. Now for $\tau \in S_n(\lambda)$, the coefficients of $x^\tau$ in $f^*_\mu$ only depends on the top homogeneous part of $f^*_\mu$, namely $f_\mu$. And we deduce from \cref{thm:main0} that 
    $[x^\tau]f^*_\mu=\delta_{\tau,\mu}.$
As in the proof of \cref{lem:V-f}, this implies that $f^*_\mu$ are linearly independent.
\end{proof}
Recall that $P_\lambda^*$ are the interpolation symmetric Macdonald polynomials defined by \cref{def:intMacdonald}.

\begin{prop}\label{prop:symmetrization}For any partition $\lambda$, we have $$P^*_\lambda=\sum_{\mu\in\SG_n(\lambda)}f^*_\mu.$$
\end{prop}
The proof is similar to the proof of \cite[Lemma 3]{CantinideGierWheeler2015} or \cite[Lemma 1.24]{CorteelMandelshtamWilliams2022}.
\begin{proof}
Let $g:=\sum_{\mu\in\SG_n(\lambda)}f^*_\mu.$ From \cref{lem:f_mu-V_lambda} we know that $g\in \VV$, and thus satisfies the vanishing conditions defining the symmetric polynomial $P_\lambda$: $g(\widetilde \rho)=0$ for any partition $|\rho|\leq |\lambda|$ with $\rho\neq \lambda$.

We now show that $g$ is symmetric. Using the relations of \cref{prop:T_i-f_star}, we show that for any $i$ we have $T_i\cdot g=tg$. This implies that $s_i g=g$ meaning that $g\in\Lambda_n^{(|\lambda|)}\cap \VV$. Hence $g$ is a scalar multiple of $P^*_\lambda$.

Finally, we know from \cite[Theorem 1.11]{CorteelMandelshtamWilliams2022} that the top homogeneous part of $g$ is 
$$\sum_{\mu\in\SG_n(\lambda)}f_\mu=P_\lambda.$$
Thus by \cref{def:intMacdonald}, $g$ must be equal to $P_{\lambda}^*$.
\end{proof}

\subsection{Characterization of interpolation ASEP polynomials}\label{ssec:interpolation_ASEP}
In this section, we give a characterization of interpolation ASEP polynomials with vanishing conditions, which thus justifies their name.

We recall that \emph{the dominance order} on partitions is the partial order such that $\lambda\leq  \mu$ if $|\lambda| < |\mu|$ or $|\lambda|=|\mu|$ and 
$$\lambda_1+\dots+\lambda_i\leq \mu_1+\dots+\mu_i,\quad \text{for any $1\leq i\leq n$}.$$

Fix  $\kappa$ and $\nu$ in $\NN^n$, and let $\lambda$ and $\mu$ be the two corresponding partitions.
We then define the partial order $<$ on $\NN^n$ such that 
$\kappa\leq \nu$ if and only if either $\lambda<\mu$ or
$$\lambda=\mu \quad \text{and} \quad\kappa_1+\dots+\kappa_i\geq \nu_1+\dots+\nu_i,\quad \text{for any $1\leq i\leq n$}.$$
We have the following triangularity property of $E^*_{\mu}$.

\begin{thm}[{\cite[Theorem 3.11]{Knop1997b}}]\label{thm:triangularity}
    Given a composition $\mu\in \NN^n$, we have
    $$E^*_\mu=x^\mu+\sum_{\nu<\mu}c_{\mu,\nu} x^\nu,$$
    for some coefficients $c_{\mu,\nu}\in\QQ(q,t)$ and where the sum is taken over compositions $\nu$ smaller than $\mu$ with respect to the partial order defined above.
\end{thm}

\begin{thm}\label{thm:characterization}
    Fix $\mu\in S_n(\lambda)$ of size $d$. Then $f^*_\mu(x_1,\dots,x_n)$ is the unique polynomial $g\in \mcP_n^{(d)}$ such that:
    \begin{itemize}
        \item for any composition $\nu$ such that $|\nu|\leq |\mu|$ and $\nu\notin S_n(\lambda)$, we have
        $g(\widetilde{\nu})=0$.
        \item for $\tau \in S_n(\lambda)$, then 
$$[x^\tau]g =\delta_{\tau,\mu}.$$
    \end{itemize}
\end{thm}
Recall that the first condition is equivalent to saying that $g\in \VV$.

\begin{proof}
    The fact that $f^*_\mu\in \VV$ was proven in \cref{lem:f_mu-V_lambda}. Now for $\tau \in S_n(\lambda)$, the coefficients of $x^\tau$ in $f^*_\mu$ only depends on the top homogeneous part of $f^*_\mu$, namely $f_\mu$. And we deduce from \cref{thm:main0} that 
    $[x^\tau]f^*_\mu=\delta_{\tau,\mu}.$

Let us now prove that $f^*_\mu$ is the unique polynomial satisfying the properties of the proposition. Let $g$ be a polynomial satisfying these properties and set $h:=f^*_\mu-g$. We want to prove that $h=0$. We have that $h\in\VV$ and that $[x^\tau]h=0$ for $\tau\in S_n(\lambda)$.
By \cref{lem:span}, we can then expand it
$h=\sum_{\tau\in S_n(\lambda)}d_\tau E^*_\tau.$
We want to prove that the coefficients $d_\tau$ are all zero. Suppose that this is not the case, and let $\kappa$ be a maximal element in the set $\{\tau\in S_n(\lambda):d_\tau\neq 0\}$.
We then have from \cref{thm:triangularity} that
$[x^\kappa]h=d_\kappa\neq 0$ which is a contradiction.
\end{proof}

\section{An algebraic recursion for \texorpdfstring{$f^*_{\mu}$}{f*}  
when \texorpdfstring{$\mu$}{the composition} is packed}\label{sec:packed}

We start this section by recalling the \emph{two-line recursion} for homogeneous ASEP polynomials established in \cite{CorteelMandelshtamWilliams2022}, see \cref{lem:F_decomposition1}; its combinatorial analogue in terms of multiline queues is in \cref{lem:recursive}.
Our goal will be to 
give an analogue of \cref{lem:F_decomposition1} for interpolation ASEP polynomials indexed by \emph{packed compositions}, see \cref{thm:recurrence}.

Given a composition $\nu$, let $\nu^-:=(\nu^-_1,\dots,\nu^-_n)$,
where
$\nu^-_i=\max(\nu_i-1,0)$.
\begin{lem}[{\cite[Lemma 3.2]{CorteelMandelshtamWilliams2022}}]\label{lem:F_decomposition1}
Fix a composition $\mu$. There exists a family of coefficients $a_\mu^\nu\in\QQ(q,t)$ such that 
\begin{equation}\label{eq:F_decomposition1}
  f_\mu(x_1,\dots,x_n)=\left(\prod_{i:\mu_i>0}x_i\right)\sum_{\nu}a^\nu_{\mu} f_{\nu^-}(x_1,\dots,x_n),  
\end{equation}
where the sum runs over compositions $\nu$ which are permutations of $\mu$ after removing the 1's from $\mu$.
\end{lem}

\begin{definition}
    For fixed $k,n$ with $1 \leq k \leq n$, we say that a composition 
    $\mu=(\mu_1,\dots,\mu_n)$ is \emph{packed} of \emph{type $(k,n)$} if 
    $\mu_i \neq 0$ for $i \leq k$ and 
    $\mu_i=0$ for $i>k$.
    Let $\Pack(k,n)$ denote the set of all packed compositions of type $(k,n)$.
\end{definition}

\begin{thm}\label{thm:recurrence}
    Let $\mu\in \Pack(k,n)$ be a packed composition. Then
    \begin{equation}\label{eq:thm_recurrence}
      f_{\mu}^{*}(x_1,\dots,x_n)=\prod_{i=1}^k (x_i -t^{-n+1})\sum_{\nu}a^\nu_{\mu} q^{|\nu^-|}f^*_{\nu^-}\left(\frac{x_1}{q},\dots,\frac{x_n}{q}\right),  
    \end{equation}
    where $a_\mu^{\nu}$ are the coefficients of \cref{eq:F_decomposition1} (see also~\cref{amulambda}). 
\end{thm}
Recall that $a_\mu^{\nu}$ is  0 unless $\nu$ is a permutation of $\mu$ after removing the 1's of $\mu$. In particular, if $\nu$ contributes to the sum of \cref{eq:thm_recurrence}, then $|\nu^-|=|\mu|-k$.

    \cref{thm:recurrence} will be proved in 3 steps:
    \begin{enumerate}[label=(Step \arabic*)]
        \item \label{first}  We prove that we can write  
\begin{equation} \label{eq:factor}
f_{\mu}^{*}(x_1,\dots,x_n) = \prod_{i=1}^k (x_i -t^{-n+1})\ Q(x_1,\dots,x_n),
\end{equation}
    where $\deg(Q(x_1,\dots,x_n)) = |\mu|-k$.
    \item \label{second} In \cref{eq:factor}, we have  $$Q(x_1,\dots,x_n) = \sum_{\nu:|\nu|=|\mu|-k}
    b_{\mu}^{\nu} \ q^{|\mu|-k} \ f_{\nu}^{*}\left(\frac{x_1}{q}, \frac{x_2}{q},\dots,\frac{x_n}{q}\right)$$
    where $b_{\mu}^{\nu} \in \QQ(q,t)$.
    
    \item \label{third} The coefficients $b_{\mu}^{\nu}$ in \ref{second} 
    are directly related to the coefficients $a_{\mu}^{\nu}$
    from
    \cref{eq:F_decomposition1}.
    More precisely, for any composition $\nu$ without parts of size 1, we have $b^{\nu^-}_{\mu}=a^{\nu}_{\mu}$.
\end{enumerate}

Before proving the theorem, we need a little bit of preparation.
We start by recalling the \emph{shape permuting
operator} from \cite[Equation (17)]{HHL3}.

\begin{prop}[\cite{HHL3}]
\label{prop:shape-homogeneous}
Let $\nu$ be a composition, and suppose $\nu_i>\nu_{i+1}$.
Write
\[
r_i(\nu)=
\#\{j<i \mid \nu_{i+1}<\nu_{j}\leq \nu_i\}
+
\#\{j>i \mid \nu_{i+1}\leq \nu_j < \nu_i\}.
\]
Then
\begin{equation}
\label{shape-permute0}
E_{s_i \nu}(\xx;q,t) = 
\left(T_i + \frac{1-t}{1-q^{\nu_{i}-\nu_{i+1}}t^{r_i(\nu)}}\right)
E_{\nu}(\xx;q,t).
\end{equation}
\end{prop}
 
\begin{lem}[{\cite[Lemma 3.1]{Knop1997b}}]\label{lem:Heckeop-evaluation}
    Fix a polynomial $f\in\mcP_n$ and a composition $\mu\in\NN^n$. Then, for any $1\leq i\leq n-1$, $(T_i f)(\Tilde{\mu})$ is a linear combination of $f(\widetilde{\mu})$ and  $f(\widetilde{s_i\mu})$.
\end{lem}

We use this lemma to give an analogue of \cref{prop:shape-homogeneous} for interpolation polynomials.
\begin{prop}\label{prop:shape}
Let $\nu$ be a composition, and suppose $\nu_i>\nu_{i+1}$. Then
    \begin{equation}
\label{shape-permute}
E^*_{s_i \nu}(\xx;q,t) = 
\left(T_i + \frac{1-t}{1-q^{\nu_{i+1}-\nu_i}t^{r_i(\nu)}}\right)
E^*_{\nu}(\xx;q,t).
\end{equation}
\end{prop}
\begin{proof}
    We start by proving that $T_i E^*_{\nu}(\xx;q,t)$ is a linear combination of $E^*_{\nu}(\xx;q,t)$ and $E^*_{s_i\nu}(\xx;q,t).$
    Since $T_i$ is a homogeneous operator, $T_i E^*_{\nu}$ has degree $|\nu|$. Moreover, using \cref{lem:Heckeop-evaluation} and the vanishing conditions satisfied by $E^*_{\nu}$
    (\cref{thm:Knop_Sahi}), we get that for any $\mu\in\NN^n$ with $|\mu| \leq |\nu|$ and $\mu\notin \{\nu,s_i\nu\}$, we have $(T_i E^*_{\nu})(\widetilde{\mu})=0$. Using now the fact that 
    $T_i E_{\nu}^*$ is a linear combination of 
    $E_{\mu}^*$ for $|\mu|=|\nu|$ (\cref{lem:stability}), and 
    \cref{rmk:Enu(nu)}, we conclude that 
    the coefficient of $E^*_{\mu}$ in $T_i E^*_{\nu}$ is $0$ for all $\mu\notin \{\nu,s_i\nu\}$.

    To get the coefficients of this linear expansion, it is enough to look at the top homogeneous part. We then conclude using \cref{prop:shape-homogeneous}.
\end{proof}
The fact that $T_i E^*_{\nu}$ is a linear combination of $E^*_{\nu}$ and $E^*_{s_i\nu}$ will be useful later. Although the explicit coefficients of this expansion will not be needed, we provide them here for completeness.

\begin{definition}\label{def:precedence-order}
 Let $\mu, \nu\in \NN^n$.  We say that \emph{$\mu$ precedes $\nu$} and write
 $\mu \preceq \nu$ if there exists $\pi\in S_n$ such that 
 \begin{itemize}
     \item $\mu_i \leq \nu_{\pi(i)}$ for all $i$,
     \item if $i>\pi(i)$, then $\mu_i < \nu_{\pi(i)}.$
 \end{itemize}
\end{definition}

\begin{example}
Consider the compositions $\mu=(3,3,2,0)$, $\nu=(5,4,1,2)$ and $\tau=(5,4,0,3)$. Then $\mu\preceq\nu$ but $\mu\npreceq\tau$.
\end{example}
The following is known as the extra \emph{vanishing property}.
\begin{thm}[{\cite[Theorem 4.5]{Knop1997b}}]\label{lem:extravanishing}
    If $\mu \npreceq \nu$ then $E_{\mu}^*(\widetilde{\nu})=0.$
\end{thm}

\begin{lem}\label{lem:2}
    Let $\mu\in \Pack(k,n)$ and let $\nu=(\nu_1,\dots,\nu_n)$ be a composition
    such that there exists $i_0\leq k$ such that $\nu_{i_0}=0$.
    Then $\mu \npreceq \nu$.  It follows that $E_{\mu}^{*}(\widetilde{\nu})=0$ and 
    $f_{\mu}^*(\widetilde{\nu})=0$.
\end{lem}
\begin{proof}
We start by proving that $\mu \npreceq \nu$. 
Assume that this is not the case, then there exists $\pi\in S_n$ as in \cref{def:precedence-order}. Let $j_0=\pi^{-1}(i_0)$. Since $\mu_{j_0}\leq \nu_{i_0}=0$ and $\mu\in\Pack(k,n)$, we get that $j_0>k$. We then have $i_0=\pi(j_0)\leq k<j_0$ but $\mu_{j_0}=\nu_{\pi(j_0)}$, which contradicts the second item in \cref{def:precedence-order}. This proves that $\mu \npreceq \nu$. 

Now by \cref{lem:extravanishing}, it follows that $E_{\mu}^*(\widetilde{\nu})=0$.
We now claim that $f_{\mu}^*\in \Span\{E_{\tau}^* \ \vert \ \tau\in \Pack(k,n)\}.$
Let $\lambda$ be the partition obtained by sorting the parts of $\mu$.
To prove the claim, recall that by \cref{def:ASEP}, since $\lambda$ is a partition,
$f_{\lambda}^* = E_{\lambda}^*.$  Now by definition, we can obtain
$f_{\mu}^*$ from $f_{\lambda}^*=E_{\lambda}^*$ by applying the $T_i$ operators for $i\leq k-1$, which by 
\cref{prop:shape} will give a linear combination of polynomials
$E_{\tau}^*$ for $\tau\in \Pack(k,n)$. 
Now by the claim, and the fact that $E_{\tau}^*(\widetilde{\nu})=0$ for 
$\tau\in \Pack(k,n)$, it follows that  $f_{\mu}^*(\widetilde{\nu})=0$.
\end{proof}

\begin{proof}[Proof of \cref{thm:recurrence} \ref{first}]
We will prove by induction on $1\leq \ell\leq k$ that 
$$f_{\mu}^{*}(x_1,\dots,x_n) = \prod_{i=1}^\ell (x_i -t^{-n+1})\ Q_{\ell}(x_1,\dots,x_n),$$
for some polynomial $Q_\ell$ of degree $|\mu|-\ell$. We then get \eqref{eq:factor} by taking $\ell=k$.

For the base case, when $\ell=1$, we start by 
writing 
\begin{equation*}
    f_{\mu}^*(x_1,\dots,x_n) = (x_1-t^{-n+1}) Q_1(x_1,\dots,x_n)+R(x_2,\dots,x_n)
\end{equation*}
for some polynomial $R(x_2,\dots,x_n)$.
Consider $\nu=(\nu_2,\dots,\nu_n)\in \NN^{n-1}.$
We will show that $R(\widetilde{\nu})=0$ for all such $\nu$, which by \cref{thm:Knop-Sahi-interpolation} will imply that 
$R=0$.
Now let $\rho=(0,\nu_2,\dots,\nu_n)$.  Then $\widetilde{\rho} = (t^{-n+1}, \widetilde{\nu}).$  By \cref{lem:2}, we have that 
$f^*_{\mu}(\widetilde{\rho})=0$.  
But also $f_{\mu}^*(\widetilde{\rho}) = R(\widetilde{\nu})$ so 
$R(\widetilde{\nu})=0$ for all $\nu\in \NN^{n-1}$, hence $R=0$.

For the induction step, suppose that \eqref{eq:factor} holds for $\ell-1$.
Thus we can write 
\begin{align}
f_{\mu}^*(x_1,\dots,x_n)&=\prod_{i=1}^{\ell-1} (x_i-t^{-n+1}) Q_{\ell-1}(x_1,\dots,x_n)\\
&=\prod_{i=1}^{\ell-1}(x_i-t^{-n+1}) \left[ (x_\ell-t^{-n+1}) Q_\ell(x_1,\dots,x_n) + 
R_\ell(x_1,\dots,\hat{x}_\ell,\dots,x_n)\right]
\label{QkRk}
\end{align}
for some polynomial $R_\ell$ in $x_1,\dots,\hat{x}_\ell,\dots,x_n$.

Let $S(x_1,\dots,x_n):=\prod_{i=1}^{\ell-1} (x_i-t^{-n+1}) R_\ell(x_1,\dots,\hat{x}_\ell,\dots,x_n).$  Clearly $R_\ell$ is identically zero if and only if $S$ is identically zero.
Let $\nu=(\nu_1,\dots,\nu_{\ell-1},\nu_{\ell+1},\dots,\nu_n)\in \NN^{n-1}$, and define
$\rho=(\nu_1,\dots,\nu_{\ell-1},0,\nu_{\ell+1},\dots,\nu_n)\in \NN^n$.

Case 1: If there exists $i<\ell$ such that $\nu_i=0$, then take the smallest such $i$.
We have that $\widetilde{\rho}_i=t^{-n+1}$ which implies that 
$S(\widetilde{\rho})=0$.

Case 2: Otherwise $\nu_i \neq 0$ for any $i<\ell$.
Then $\widetilde{\rho} = (\widetilde{\nu}_1,\widetilde{\nu}_2,\dots,t^{-n+1},
\widetilde{\nu}_{\ell+1},\dots,).$
Now we have that $S(\widetilde{\nu})$ is a multiple of 
$R_\ell(\widetilde{\nu})$, and from \eqref{QkRk} we have that 
$R_\ell(\widetilde{\nu})$ is a non zero multiple of $f_{\mu}^*(\widetilde{\rho})$. We use here the fact $\widetilde{\nu}_i\neq t^{-n+1}$ for $1\leq i\leq \ell-1$.
Finally from \cref{lem:2} we have 
$f_{\mu}^*(\widetilde{\rho})=0$.
This shows that for any $\nu\in \NN^{n-1}$ we have that 
$R_\ell(\widetilde{\nu})=0$.  This shows that $R_\ell$ must be identically zero, so we are done.
\end{proof}

The following lemma will be helpful in \ref{second} of the proof of \cref{thm:recurrence}.
\begin{lem}\label{lem:vanishing}
    Let $g(x_1,\dots,x_n)$ be a polynomial in $x_1,\dots,x_n$.
    Then $g(\widetilde{\nu})=0$ for all 
    $|\nu|\leq k$ if and only if 
the coefficient    $[E_{\nu}^*] g$ of $E_{\nu}^*$ in $g(x_1,\dots,x_n)$ is $0$ for all $|\nu| \leq k$.
\end{lem}
\begin{remark}\label{rem:vanishing}
Since the families $\left(E_{\nu}^*\right)_{|\nu|\leq d}$ and $\left(f_{\nu}^*\right)_{|\nu|\leq d}$ are both bases of the space of polynomials of degree at most $d$ (see \cref{lem:span} and \cref{prop:V*-f*}),
the two conditions in \cref{lem:vanishing} are equivalent to the condition that 
the coefficient    $[f_{\nu}^*] g$ is $0$ for all $|\nu| \leq k$.
\end{remark}
\begin{proof}
For the forward direction,  we will use induction on $k$;
suppose that the forward direction of the lemma is true for $k$.
Now suppose that $g(\widetilde{\nu})=0$ for all 
    $|\nu|\leq k+1$.  By the induction hypothesis,
         $[E_{\mu}^*] g=0$ for all $|\mu| \leq k$.
         Thus we can write 
         \begin{equation}
         \label{expandg}    
         g=\sum_{i=k+1}^m \sum_{\mu \vdash i} a_\mu E_{\mu}^*,
         \end{equation}
where $m=\deg(g)$.
Now for $\nu \vdash k+1$, we have that 
\begin{equation*}
    0=g(\widetilde{\nu})= \sum_{i=k+1}^m \sum_{\mu \vdash i} a_\mu E_{\mu}^*(\widetilde{\nu})
     = a_{\nu} E_{\nu}^*(\widetilde{\nu}),
\end{equation*}
where in the last equality we used 
\cref{thm:Knop_Sahi}.  But now since $E_{\nu}^*(\widetilde{\nu}) \neq 0$ (see \cref{rmk:Enu(nu)}), it follows that $a_{\nu}=0$.

For the backward direction, suppose that 
   $[E_{\nu}^*] g=0$
 for all $|\nu| \leq k$.  Then as before we can write
 $g$ as in \eqref{expandg}.  By \cref{thm:Knop_Sahi},
 for all $\nu$ such that $|\nu|\leq k$ and $|\mu|>k$, we have
 $E_{\mu}^*(\widetilde{\nu})=0$.  But now by \eqref{expandg}, we have that 
 $g(\widetilde{\nu})=0$.
\end{proof}

\begin{proof}[Proof of \cref{thm:recurrence} \ref{second}]
Let $\widehat{Q}(x_1,\dots,x_n):=Q(qx_1,qx_2,\dots,qx_n)$.
Since $\widehat{Q}$ is a polynomial of degree $|\mu|-k$ which lies in the space spanned by
$f_{\nu}^*$ for $|\nu|\leq |\mu|-k$, it follows that we can write
$$\widehat{Q}(x_1,\dots,x_n) = \sum_{\nu, |\nu| \leq |\mu|-k}
    b_{\mu}^{\nu}  \ f_{\nu}^{*}({x_1}, {x_2},\dots,{x_n}),$$
    where $b_{\mu}^{\nu} \in \QQ(q,t).$ 
    We want to show that $[f_{\nu}^*]\widehat{Q}=0$ for $|\nu|<|\mu|-k$.
    By \cref{rem:vanishing}, it suffices to show that 
    $\widehat{Q}(\widetilde{\rho})=0$ for all $\rho$ with
    $|\rho|<|\mu|-k$.

    Choose $\rho$ such that 
    $|\rho|<|\mu|-k$.  Let $\rho^+ = (\rho_1+1,\rho_2+1,\dots,\rho_n+1).$
    It is clear from the definitions that $\widehat{Q}(\widetilde{\rho}) = Q(\widetilde{\rho^+}).$
    Note also that $\widetilde{\rho^+}$ has no entries of the form
    $t^i$.
    From \ref{first}, we have 
    $$
f_{\mu}^{*}(x_1,\dots,x_n) = \prod_{i=1}^k (x_i -t^{-n+1})\ Q(x_1,\dots,x_n),$$ which implies that 
$f_{\mu}^*(\widetilde{\rho^+})$ is a nonzero multiple of 
$Q(\widetilde{\rho^+})$.

We now claim that for any $\nu$ such that $|\nu| < |\mu|-k$, we have
$\mu \npreceq \nu^+$.  To prove the claim, assume that 
$\mu \preceq \nu^+$.  Then there exists some permutation $\pi$ such that $\mu_i \leq \nu_{\pi(i)}^+$ for $1\leq i \leq k$,
and $0=\mu_i \leq \nu_{\pi(i)}^+$ for $k+1 \leq i \leq n$.
The sum of the $\nu_{\pi(i)}^+$ for $k+1 \leq i \leq n$ is at least $n-k$, which implies that the sum of the $\nu_{\pi(i)}^+$ for $1\leq i \leq k$ is at most $|\nu|+k$.  But this implies that 
$|\mu| \leq |\nu|+k$, which is a contradiction.

Now for our chosen $\rho$, since $|\rho|<|\mu|-k$, we have that 
$\mu \npreceq \rho^+$.  But now by \cref{lem:extravanishing}, it follows that
$f_{\mu}^*(\widetilde{\rho^+})=0$ for all $\rho$ with $|\rho|<|\mu|-k$, and since 
$f_{\mu}^*(\widetilde{\rho^+})$ is a nonzero multiple of 
$Q(\widetilde{\rho^+})$, it follows that 
$Q(\widetilde{\rho^+})=\widehat{Q}(\widetilde{\rho}) = 0$
for all $\rho$ with $|\rho|<|\mu|-k$.  We have thus proved that 
$$\widehat{Q}(x_1,\dots,x_n) = \sum_{\nu, |\nu| = |\mu|-k}
    b_{\mu}^{\nu}  \ f_{\nu}^{*}({x_1}, {x_2},\dots,{x_n}),$$
     and hence 
    $${Q}(x_1,\dots,x_n) = \sum_{\nu, |\nu| = |\mu|-k}
    b_{\mu}^{\nu}  \ f_{\nu}^{*}\left(\frac{x_1}{q}, \frac{x_2}{q},\dots,\frac{x_n}{q}\right).$$
    But now by renaming the notation $b_{\mu}^{\nu}$ by $b_{\mu}^{\nu} \ q^{|\mu|-k}$, which is a convenient notation for 
    \ref{third}, we get the desired result.    
\end{proof}

\begin{proof}[Proof of \cref{thm:recurrence} \ref{third}]
Note that, since $\deg(f^*_\mu)=|\mu|$, the transformation 
$$f^*_\mu\longmapsto q^{|\mu|}f^*_\mu\left(\frac{x_1}{q}, \frac{x_2}{q},\dots,\frac{x_n}{q}\right)$$
does not change the top homogeneous part of the polynomial.
\ref{third} follows then from \ref{first} and \ref{second} of \cref{thm:recurrence}, by 
looking at the top homogeneous part of \eqref{eq:factor}, and using the fact that in 
\ref{second}, the sum is over compositions $\nu$ such that $|\nu| = |\mu|-k$. We also use here \cref{lem:basis}, which says  that $(f_\nu)_{\nu\vdash |\mu|-k}$ is a basis of the space of polynomials of degree $|\mu|-k$.
\end{proof}

\section{An algebraic recursion for \texorpdfstring{$f_{\mu}^*$}{f*} indexed by arbitrary compositions}\label{sec:tworow}

The main goal of the next two sections is to finish the proof of \cref{thm:main}. 
There are two main steps, the first of which is algebraic while the second is combinatorial:
\begin{itemize}
    \item  We start from the recursion given for the interpolation ASEP polynomials $f^*_\mu$  in \cref{thm:recurrence}, when $\mu$ is a  packed composition. 
    By applying Hecke operators to this recursion, we generalize it to any composition $\mu$. This recursion involves a family of coefficients $(b^\alpha_\mu)$ defined in \cref{def:b_recursion},  and encoded by the action of the Hecke operators on a variant $f_\alpha$ of the ASEP polynomials, indexed by signed compositions. This step corresponds to \cref{thm:f_decomposition}.
    \item We show that the generating function of one single signed layer satisfies the same recursion as the coefficients $(b^\alpha_\mu)$, see \cref{prop:recursion}. Thus, we show that the \emph{algebraic recursion} for the polynomials $f^*_\mu$ corresponds to a \emph{combinatorial recursion} for signed MLQs. The combinatorial recursion encodes the fact that if we remove the bottom signed and the bottom classic layers of a signed MLQ with $2L$ rows, we obtain a signed MLQ with $2(L-1)$ rows, see \cref{lem:F_decomposition}. The main theorem is then obtained by induction on the number of rows.
\end{itemize}

\subsection{Some preliminaries about Hecke operators}
\begin{lem}
    Fix a polynomial in $n$ variables $A\in\mcP_n$ and let $1\leq i\leq n-1$. Then
    \begin{align}
      &T_i(x_ix_{i+1}A)=x_ix_{i+1}T_i(A)\label{eq:Hecke1},\\
      &T_i(x_i A)=x_{i+1}T_i(A)+(1-t)x_{i+1}A,\label{eq:Hecke2}\\
      &T_i(x_{i+1}A)=x_{i}T_i(A)-(1-t)x_{i+1} A.\label{eq:Hecke3}
    \end{align}
\end{lem}

\begin{proof}
We have
\begin{align*}
  T_i(x_ix_{i+1}A)
  &=tx_ix_{i+1}A-\frac{tx_i-x_{i+1}}{x_i-x_{i+1}}(x_ix_{i+1}A-s_i(x_ix_{i+1}A))\\
  &=x_ix_{i+1}T_i(A).
\end{align*}
This gives \cref{eq:Hecke1}. Notice that, more generally, $T_i(BA)=BT_i(A)$ for any polynomial $B$ which is symmetric in $x_i$ and $x_{i+1}$. We now prove \cref{eq:Hecke2}

\begin{align*}
  T_i(x_iA)
  &=tx_iA-\frac{tx_i-x_{i+1}}{x_i-x_{i+1}}(x_iA-s_i(x_iA))\\
  &=tx_iA-\frac{tx_i-x_{i+1}}{x_i-x_{i+1}}(x_iA-x_{i+1}s_i(A))\\
  &=x_{i+1}T_i(A)+(1-t)x_{i+1} A.
\end{align*}
We obtain similarly \cref{eq:Hecke3}
\begin{align*}
  T_i(x_{i+1}A)
  &=tx_{i+1}A-\frac{tx_i-x_{i+1}}{x_i-x_{i+1}}(x_{i+1}A-s_i(x_{i+1}A))\\
  &=tx_{i+1}A-\frac{tx_i-x_{i+1}}{x_i-x_{i+1}}(x_{i+1}A-x_{i}s_i(A))\\
  &=x_{i}T_i(A)-(1-t)x_{i+1}A.\qedhere
\end{align*}
\end{proof}
Combining \cref{eq:Hecke1,eq:Hecke2,eq:Hecke3}, we obtain that if a polynomial $A$ is divisible by $x_i$, then
\begin{equation}\label{eq:Hecke4}
  T_i(A/x_i)=\frac{1}{x_{i+1}}T_i(A)-(1-t)\frac{A}{x_i},  
\end{equation}
and if it is divisible by $x_{i+1}$ then
\begin{equation}\label{eq:Hecke5}
  T_i(A/x_{i+1})=\frac{1}{x_{i}}T_i(A)+(1-t)\frac{A}{x_i}.
\end{equation}
Finally, if $A$ is divisible by $x_{i}x_{i+1}$ then
\begin{equation}\label{eq:Hecke6}
    T_i(A/x_ix_{i+1})=\frac{1}{x_ix_{i+1}}T_i(A).
\end{equation}

\subsection{Action of Hecke operators on extended ASEP polynomials}
For any composition $\mu$, we define the polynomial
$$\hatfstar{\mu}:=q^{|\mu|}f^*_{\mu}\left(\frac{x_1}{q},\dots,\frac{x_n}{q}\right).$$

The Hecke operators act in the same way on three versions of ASEP polynomials.

\begin{lem}\label{lem:Hecke_ASEP}
  For a composition $\mu$, and for $g\in\{f,f^*,\hatfstar{}\}$, we have
    \begin{align}
 T_i(g_\mu)=\left\{
 \begin{array}{ll}
 g_{s_i\mu}   & \text{ if $\mu_i>\mu_{i+1}$,} \\
 t g_{\mu}    & \text{ if $\mu_i=\mu_{i+1}$,}\\
 tg_{s_i\mu}-(1-t)g_\mu& \text{ if $\mu_i<\mu_{i+1}$.}
 \end{array}     
 \right.\label{eq:HeckeASEP}
\end{align}
\end{lem}
\begin{proof}
The result is known for the functions $f_\mu$ (see \cref{rem:ASEPpoly}) and for the functions $f^*_\mu$ (see~\cref{prop:T_i-f_star}). Let us check it for $\hatfstar{\mu}$. First notice that the linear map
$$\phi_r:h(x_1,\dots,x_n)\longmapsto q^r h\left(\frac{x_1}{q},\dots,\frac{x_n}{q}\right)$$
acts diagonally on a homogeneous function $h$ of degree $m$: $\phi_r(h)=q^{r-m}h$. Since the operators $T_r$ are homogeneous, we have $T_i\circ\phi_r (h)=\phi_r\circ T_i (h).$
We now write 
$T_i\left(\hatfstar{\mu}\right)=T_i\circ \phi_{|\mu|}(f^*_\mu)=\phi_{|\mu|}\circ  T_i(f^*_\mu)$. We conclude using the fact that $T_i(f^*_\mu)$ is a linear combination of $f^*_\mu$ and $f^*_{s_i\mu}$, and that $|s_i\mu|=|\mu|$. 
\end{proof}

We recall from \cite{CorteelMandelshtamWilliams2022} that the homogeneous ASEP polynomial $f_\mu$ is divisible by $\prod_{i:\mu_i>0}x_i$ (see also \cref{lem:F_decomposition1}), which corresponds to the weights of the balls in Row~1. In \cref{def:extendedASEP} below we extend the definition of ASEP polynomials to all \emph{signed} compositions; here we assign a weight of $t^{-n+1}$ to ``negative'' balls.
 
For $\alpha\in\ZZ^n$, set 
\begin{equation}\label{eq:absolute}
\lVert\alpha\rVert:=(|\alpha_1|,\dots,|\alpha_n|).\end{equation}
\begin{definition}\label{def:extendedASEP}
   Fix $\alpha\in\ZZ^n$.
    We define the \emph{extended ASEP polynomial} 
    $$f_\alpha:=\frac{f_{\lVert \alpha\rVert}}{\prod_{i:\alpha_i<0}(-t^{n-1}x_i)}.$$
\end{definition}

\begin{prop}\label{prop:Hecke_extendedASEP}
Given $\alpha\in \ZZ^n$ and $1\leq i\leq n-1$, we have the following action of $T_i$ on $f_\alpha$.

\begin{enumerate}[font=\normalfont]
\item Case $\alpha_i,\alpha_{i+1}\geq 0$ or $\alpha_i,\alpha_{i+1}< 0$: we have 
\begin{align*}
 T_i(f_\alpha)=\left\{
 \begin{array}{ll}
 f_{s_i\alpha}   & \text{ if $\alpha_i>\alpha_{i+1}\geq 0$, or $-\alpha_i>-\alpha_{i+1}> 0$} \\
 t f_{\alpha}    & \text{ if $\alpha_i=\alpha_{i+1}\geq 0$, or $-\alpha_i=-\alpha_{i+1}> 0$}\\
 tf_{s_i\alpha}-(1-t)f_\alpha& \text{ if $\alpha_{i+1}>\alpha_{i}\geq 0$, or $-\alpha_{i+1}>-\alpha_{i}> 0$}
 \end{array}     
 \right.
\end{align*}

\item Case $\alpha_i\geq 0$ and $\alpha_{i+1}<0$: we have
\begin{align*}
 T_i(f_\alpha)=\left\{
 \begin{array}{ll}
 f_{s_i\alpha}+(1-t)f_{\alpha_1,\dots,-\alpha_i,-\alpha_{i+1},\dots}   & \text{ if $\alpha_i>-\alpha_{i+1}> 0$,} \\
 f_{s_i\alpha}    & \text{ if $\alpha_i=-\alpha_{i+1}> 0$,}\\
 tf_{s_i\alpha}& \text{ if $-\alpha_{i+1}>\alpha_{i}\geq 0$.}
 \end{array}     
 \right.
\end{align*}
\item Case $\alpha_i< 0$ and $\alpha_{i+1}\geq0$: we have
\begin{align*}
 T_i(f_\alpha)=\left\{
 \begin{array}{ll}
 f_{s_i\alpha}-(1-t)f_{\alpha}   & \text{ if $-\alpha_i>\alpha_{i+1}\geq 0$,} \\
 tf_{s_i\alpha}-(1-t)f_{\alpha}    & \text{ if $-\alpha_i=\alpha_{i+1}>0$,}\\
 tf_{s_i\alpha}-(1-t)\left(f_\alpha+f_{\alpha_1,\dots,-\alpha_i,-\alpha_{i+1},\dots}\right)& \text{ if $\alpha_{i+1}>-\alpha_{i}>0$.}
 \end{array}     
 \right.
\end{align*}

\end{enumerate}
\end{prop}

\begin{proof}
We start by proving the case $(1)$. When $\alpha_i,\alpha_{i+1}\geq 0$, we have

    $$T_i(f_\alpha)=T_i\cdot\prod_{j:\alpha_j<0}\frac{1}{-t^{n-1}x_j}f_{\lVert\alpha\rVert}
        =\prod_{j:\alpha_j<0}\frac{1}{-t^{n-1}x_j}T_i\left( f_{\lVert\alpha\rVert}\right).$$
    We use here the fact that $\prod_{j:\alpha_j<0}\frac{1}{-t^{n-1}x_j}$ is independent from $x_i$ and $x_{i+1}$ (and is in particular symmetric in these variables).
The result follows then from the action of $T_i$ on the (non extended) ASEP polynomials \cref{eq:HeckeASEP}.

Now if $\alpha_i,\alpha_{i+1}<0$, then
$$T_i(f_\alpha)=(-t)^{-2n+2}T_i(f_{\dots,-\alpha_i,-\alpha_{i+1},\dots}/(x_ix_{i+1}))=(-t)^{-2n+2}T_i(f_{\dots,-\alpha_i,-\alpha_{i+1},\dots})/(x_ix_{i+1})$$
by \cref{eq:Hecke6}.
Since $-\alpha_i,-\alpha_{i+1}> 0$, we can use the equations proved above to conclude.

Similarly, the other cases are obtained from $(1)$ using \cref{eq:Hecke4,eq:Hecke5}. For example, let us check the case (2), i.e when $\alpha_i\geq 0$ and $\alpha_{i+1}<0$. We have
 $$T_i(f_\alpha)
        =(-t)^{-n+1}T_i\left( f_{\dots,\alpha_i,-\alpha_{i+1},\dots}/x_{i+1}\right).$$
        Applying \cref{eq:Hecke5}, we get
    $$T_i(f_\alpha)=\frac{(-t)^{-n+1}}{x_i}T_i(f_{\dots,\alpha_i,-\alpha_{i+1},\dots})+\frac{(-t)^{-n+1}(1-t)}{x_i}f_{\dots,\alpha_i,-\alpha_{i+1},\dots}.$$
    We now use case (1):
    \begin{itemize}
        \item If $\alpha_i>-\alpha_{i+1}>0$, then
       $$T_i(f_\alpha)=\frac{(-t)^{-n+1}}{x_i} f_{\dots,-\alpha_{i+1},\alpha_i,\dots}+\frac{(-t)^{-n+1}(1-t)}{x_i}f_{\dots,\alpha_i,-\alpha_{i+1},\dots}=f_{s_i\alpha}+(1-t)f_{\dots,-\alpha_i,-\alpha_{i+1},\dots}.$$
       \item If $\alpha_i=-\alpha_{i+1}>0$, then
       \begin{align*}
         T_i(f_\alpha)&=(-t)^{-n+1}\frac{t}{x_i}f_{\dots,\alpha_i,\alpha_i,\dots}+(-t)^{-n+1}\frac{1-t}{x_i}f_{\dots,\alpha_i,\alpha_i,\dots}\\
         &=(-t)^{-n+1}\frac{1}{x_i}f_{\dots,\alpha_i,\alpha_i,\dots}=f_{s_i\alpha}.
       \end{align*}
       
        \item If $-\alpha_{i+1}>\alpha_{i+1}>0$, then
        \begin{align*}
      T_i(f_\alpha)
      &=\frac{(-t)^{-n+1}}{x_i}tf_{\dots,-\alpha_{i+1},\alpha_i,\dots}-(-t)^{-n+1}\frac{1-t}{x_i}f_{\dots,\alpha_i,-\alpha_{i+1},\dots}+(-t)^{-n+1}\frac{1-t}{x_i}f_{\dots,\alpha_i,-\alpha_{i+1},\dots}\\
      &=tf_{\dots,\alpha_{i+1},\alpha_i,\dots}
        \end{align*}
    as desired.
    \end{itemize}
We leave the proof of case (3) as an exercise.
\end{proof}

\begin{definition}\label{def:matrix}
For $1\leq i\leq n-1$, we define the infinite matrix $\Ni=(\Ni_{\alpha,\beta})_{\alpha,\beta\in \ZZ^n}$  as the matrix with entries in $\ZZ[t]$ which encodes the action of $T_i$ on the polynomials $f_\alpha$ as given in \cref{prop:Hecke_extendedASEP}:
$$T_i(f_\alpha)=\sum_{\beta}\Ni_{\alpha,\beta} f_\beta.$$
\end{definition}
For example, if $\alpha_i>-\alpha_{i+1}>0$ then
$\Ni_{\alpha,(\dots,-\alpha_i,-\alpha_{i+1},\dots)}=1-t.$

This matrix is quasi-diagonal: if $\beta\notin\{\alpha,s_i\alpha,(\dots,-\alpha_i,-\alpha_{i+1},\dots)\}$, then $\Ni_{\alpha,\beta}=0$.

\subsection{The polynomials \texorpdfstring{$h_\alpha$}{h}}

Recall the definition of $a_{\mu}^{\lambda}$ from  \cref{sec:2line}. 
Also recall from \eqref{eq:absolute} that 
$\lVert\alpha\rVert:=(|\alpha_1|,\dots,|\alpha_n|).$
For any $\alpha\in \ZZ^n$, we set
$$\wt_\alpha:=\prod_{i,\alpha_i>0}x_i\prod_{i,\alpha_i<0}\frac{-1}{t^{n-1}}.$$

Using \cref{eq:F_decomposition1} 
we write 
$$f_\mu:=\sum_{\lambda\in \NN^n}\wt_\mu a^\lambda_{\mu} f_{\lambda^-}.$$
The polynomials $f_\alpha$ from \cref{def:extendedASEP} can then be written
\begin{equation}
    \label{eq:extendedASEP}
f_\alpha=\sum_{\lambda\in \NN^n}\wt_\alpha a^\lambda_{\|\alpha\|} f_{\lambda^-}.
\end{equation}

We now define a family of polynomials $h_\alpha$ which can be thought of as an intermediate step between the homogeneous polynomials $f_\mu$ and the interpolation polynomials $f^*_\mu$.
\begin{definition}
    Given $\alpha\in \ZZ^n$, we define
    \begin{equation}\label{eq:def_h}
      h_\alpha:=\sum_{\lambda\in \NN^n} \wt_{\alpha} a^\lambda_{\|\alpha\|}\hatfstar{\lambda^-}.
    \end{equation}
\end{definition}

\begin{lem}\label{lem:Hecke_h}
    The action of the operator $T_i$ on $h_\alpha$, is the same as its action on $f_\alpha$. In other words,
    $$T_i(h_\alpha)=\sum_{\beta\in\ZZ^n}\Ni_{\alpha,\beta} h_\beta.$$
\end{lem}
\begin{proof}
We start by proving the result when $\alpha_i,\alpha_{i+1}\leq 0$. First, notice that in this case $\wt_\alpha$ is independent from $x_i$ and $x_{i+1}$. We then have
    \begin{align*}
      T_i(h_\alpha)
      &=T_i\left(\sum_{\lambda\in \NN^n}\wt_\alpha a^\lambda_{\lVert \alpha\rVert} \hatfstar{\lambda^-}\right),\\
      &=\sum_{\lambda\in \NN^n}\wt_{\alpha} a^\lambda_{\lVert \alpha\rVert} T_i\left(\hatfstar{\lambda^-}\right),
    \end{align*}
    since $a_{\lVert \alpha\rVert}^\lambda=a_{\lVert \alpha\rVert}^\lambda(q,t)$ is independent from the variables $x_j$. Hence, using \cref{lem:Hecke_ASEP} and \cref{def:matrix}, we have
    \begin{align}\label{eq:Ti_falpha*}
        T_i(h_\alpha)&=\sum_{\lambda,\nu\in \NN^n}\wt_{\alpha} a^\lambda_{\lVert \alpha\rVert} \Ni_{\lambda^-,\nu} \hatfstar{\nu}.
    \end{align}
    
    We now compute $T_i(f_\alpha)$ in two different ways, obtained by applying \cref{def:matrix} and \eqref{eq:extendedASEP} in one order or the other. On the one hand, we have
    \begin{equation}\label{eq:Ti_falpha1}
T_i(f_{\alpha})=\sum_{\beta\in\ZZ^n}\Ni_{\alpha,\beta}f_{\beta}=\sum_{\beta\in\ZZ^n,\kappa\in\NN^n}\Ni_{\alpha,\beta}\wt_{\beta}a_{\lVert \beta\rVert}^\kappa f_{\kappa^-}.  
    \end{equation}
    On the other hand,
\begin{align}\label{eq:Ti_falpha2}
        T_i(f_\alpha)&=T_i\left(\sum_{\lambda\in \NN^n} \wt_{\alpha} a^\lambda_{\lVert \alpha\rVert} f_{\lambda^-}\right)
        =\sum_{\lambda,\nu\in \NN^n}\wt_{\alpha} a^\lambda_{\lVert \alpha\rVert} \Ni_{\lambda^-,\nu} f_{\nu}.
    \end{align}
Notice that with the assumption $\alpha_i,\alpha_{i+1}\leq 0$, the coefficient $\Ni_{\alpha,\beta}$ is zero unless $\beta\in\{\alpha,s_i\alpha\}$ (see \cref{prop:Hecke_extendedASEP} item (1)). In particular, we have $\wt_{\beta}=\wt_\alpha$.
We can then divide \cref{eq:Ti_falpha1} and \cref{eq:Ti_falpha2} by $\wt_\alpha$, and we compare the coefficients of $f_\nu$ in the two equations (recall that ASEP polynomials are a basis by \cref{lem:basis}). We get
\begin{equation}\label{eq:coeff_fnu}
  \sum_{\kappa:\kappa^-=\nu}\ \sum_{\beta\in\ZZ^n}\Ni_{\alpha,\beta}a_{\lVert \beta\rVert}^{\kappa}
=\sum_{\lambda\in \NN^n}a^\lambda_{\lVert \alpha\rVert} \Ni_{\lambda^-,\nu},  
\end{equation}
 Injecting this into \cref{eq:Ti_falpha*}, we get 
\begin{align*}
  T_i(h_{\alpha})
  &=\sum_{\nu\in \NN^n}\wt_{\alpha}\hatfstar{\nu}\sum_{\lambda\in \NN^n} a^\lambda_{\lVert \alpha\rVert} \Ni_{\lambda^-,\nu}=\sum_{\beta\in\ZZ^n}\Ni_{\alpha,\beta}\sum_{\nu\in\NN^n}\ \sum_{\kappa:\kappa^-=\nu} a_{\lVert \beta\rVert}^{\kappa} \wt_{\beta}\hatfstar{\nu} \\
  &=\sum_{\beta\in\ZZ^n}\Ni_{\alpha,\beta}\sum_{\kappa\in \NN^n} a_{\lVert \beta\rVert}^{\kappa} \wt_{\beta}\hatfstar{\kappa^-}
=\sum_{\beta}\Ni_{\alpha,\beta}h_\beta.  
\end{align*}
This finishes the proof of the lemma in the case $\alpha_i,\alpha_{i+1}\leq 0$.

The other cases can be derived from this one using \cref{eq:Hecke1,eq:Hecke2,eq:Hecke3} (just as the cases in \cref{prop:Hecke_extendedASEP} are derived from the case $\alpha_i,\alpha_{i+1}\geq 0$). We leave this as an exercise.
\end{proof}

\subsection{Recursive decomposition for the polynomials \texorpdfstring{$f^*_\mu$}{f*}}

\begin{definition}\label{def:b_recursion}
       Let $(b^\alpha_{\mu})_{\mu\in\NN^n,\alpha\in\ZZ^n}$ 
    be the family of coefficients satisfying the following properties:
    \begin{enumerate}
        \item If $\mu\in\Pack(k,n)$ for some $k\leq n$, then
        $$b^\alpha_{\mu}=\delta_{\mu,\lVert \alpha\rVert}.$$
        \item Given $1\leq i\leq n-1$ such that $\mu_i>0$ and $\mu_{i+1}=0$, we have
        $$b^\alpha_{s_i \mu}=\sum_{\beta\in \ZZ^n}b^\beta_{\mu}\Ni_{\beta,\alpha} .$$
    \end{enumerate}
\end{definition}

It is clear from the definition that if such a family $(b^\alpha_{\mu})$ exists then it is unique. The existence will be proven combinatorially in \cref{prop:recursion}.

\begin{remark}\label{rmk:nonvanishing_condition_b}
    Note that for the family $(b^\alpha_\mu)$ satisfying the recursion of \cref{def:b_recursion}, the coefficient $b^\alpha_\mu$ is 0 unless $\alpha$ is a signed permutation of $\mu$, i.e there exists a permutation $\sigma\in S_n$ and a choice of signs $\epsilon_1,\dots,\epsilon_n\in\{\pm 1\}$ such that 
    $\alpha=(\epsilon_1\sigma(\mu_1),\dots,\epsilon_n \sigma(\mu_n))$. This can be obtained by induction on $\mu$ and using the fact that $\Ni_{\alpha,\beta}$ is 0 unless $\alpha$ is a signed permutation of $\beta$.
\end{remark}

\begin{thm}\label{thm:f_decomposition}
Let $(b^\alpha_{\mu})$ be the family of coefficients satisfying the recursion of \cref{def:b_recursion}. Define the polynomials $f^{*\lambda}_\mu$ by
$$f^{*\lambda}_\mu=f^{*\lambda}_\mu(x_1,\dots,x_n;q,t):=\sum_{\alpha\in\ZZ^n}b^\alpha_{\mu}\wt_{\alpha}a^\lambda_{\lVert \alpha\rVert}.$$
We then have 
$$f^*_\mu=\sum_{\lambda}f^{*\lambda}_{\mu}(x_1,\dots,x_n;q,t)q^{|\lambda^-|}f^*_{\lambda^-}\left(\frac{x_1}{q},\dots,\frac{x_n}{q};q,t\right)=\sum_{\lambda}f^{*\lambda}_{\mu}\widehat{f^*_{\lambda^-}}.$$
\end{thm}

\begin{proof}
We start from the packed case and we proceed by induction. If $\mu\in\Pack(k,n)$, then from \cref{def:b_recursion} item (1) we have 
\begin{align*}
f^{*\lambda}_\mu
=\sum_{\alpha:\lVert \alpha\rVert=\mu}\wt_{\alpha}a^\lambda_{\lVert \alpha\rVert}=a^\lambda_{\mu}\prod_{i=1}^k\left(x_i-t^{-n+1}\right).
\end{align*}
But we know from \cref{thm:recurrence} that
$$f^*_\mu=\prod_{i=1}^k\left(x_i-t^{-n+1}\right)q^{|\mu|-k}\sum_{\nu}a_\mu^\nu f^*_{\nu^-}\left(\frac{x_1}{q},\dots,\frac{x_n}{q};q,t\right).$$
This gives the theorem for packed compositions. We now assume that the result holds for $\mu$, and we fix $1\leq i\leq n-1$ such that $\mu_i>0$ and $\mu_{i+1}=0$. Let us prove it for $s_i\mu$.
We have from \cref{prop:T_i-f_star} item (1) that
$f^*_{s_i\mu}=T_i f^*_\mu.$
Using the recursion assumption we get
\begin{align*}
  f^*_{s_i\mu}
  =T_i \left(\sum_{\alpha\in\ZZ^n}b^\alpha_{\mu}\sum_{\lambda}\wt_{\alpha}a^\lambda_{\lVert \alpha\rVert}\widehat{f^*_{\lambda^-}}\right)=\sum_{\alpha\in\ZZ^n}b^\alpha_{\mu} T_i \left(h_\alpha\right)  
\end{align*}
where we used the definition of the polynomials $h_\alpha$ (see \eqref{eq:def_h}). Using \cref{lem:Hecke_h}, we get
$$f^*_{s_i\mu}=\sum_{\alpha\in\ZZ^n}b^\alpha_{\mu}\sum_{\beta\in \ZZ^n}\Ni_{\alpha,\beta} h_\beta=\sum_{\beta\in \ZZ^n}h_\beta\sum_{\alpha\in\ZZ^n}b^\alpha_{\mu}\Ni_{\alpha,\beta} .$$
Finally, item (2) of \cref{def:b_recursion} gives
\begin{align*}
  f^*_{s_i\mu}
  &=\sum_{\beta\in \ZZ^n}b_{s_i\mu}^\beta h_\beta \\
&=\sum_{\beta\in \ZZ^n}b_{s_i\mu}^\beta\sum_{\lambda} \wt_{\beta}a_{\lVert \beta\rVert}^\lambda \widehat{f^*_{\lambda^-}}\\
&=\sum_{\lambda} \widehat{f^*_{\lambda^-}}
\sum_{\beta\in \ZZ^n}b_{s_i\mu}^\beta\wt_{\beta}a_{\lVert \beta\rVert}^\lambda\\
&=\sum_{\lambda} \widehat{f^*_{\lambda^-}} f^{*\lambda}_{s_i\mu}
\end{align*}
which finishes the proof of the theorem.
\end{proof}

\section{Two-line queues and the proof of the main theorem}\label{sec:proof}

In this section, after introducing the notion of 
two-line  queues and 
two-line signed  queues, we will
complete the proof of the main theorem.

\subsection{Generalized two-line queues} \label{sec:2line}

We start by reviewing
the notion of \emph{generalized two-line queue} from 
\cite{CorteelMandelshtamWilliams2022} as well as a recurrence for ASEP polynomials.
This recurrence is  based on the fact that we can view a
multiline queue $Q$ with $L$ rows as a multiline queue $Q'$ with $L-1$ rows (the restriction of $Q$ to rows
$2$ through $L$) sitting on top of
a \emph{generalized} multiline queue $Q_0$ with $2$ rows (the restriction of $Q$ to rows $1$ and $2$).
Since $Q'$ occupies rows
$2$ through $L$ and has balls labeled $2$ through $L$, we identify $Q'$ with a
multiline queue obtained by decreasing the row labels and ball labels in the top $L-1$ rows of $Q$
by $1$. 
(Holes, represented by $0$,
remain holes.)
If the bottom row of $Q'$ is the composition $\lambda$, then after decreasing labels as above,
the new bottom row is $\lambda^-=(\lambda^-_1,\dots,\lambda^-_n)$,
where
$\lambda^-_i=\max(\lambda_i-1,0)$.
Meanwhile $Q_0$ has just two rows, but its balls are labeled $1$ through $L$;
we refer to it as a \emph{generalized two-line queue}.

\begin{definition}\label{def:two-row}
A \emph{generalized two-line queue} is a two-row multiline queue whose top and bottom rows are represented by a pair of compositions $\lambda,\mu \in \NN^n$,  satisfying the following conditions: $\lambda$ has no parts of size 1, and for each $j>1$, $\#\{i:\mu_i=j\}=\#\{i:\lambda_i=j\}$. Moreover, for each $i$, either $\mu_i=0$, or $\lambda_i \leq \mu_i$. (In other words, a larger label cannot be directly above a smaller nonzero label, as in a usual multiline queue.) 

For $\mu \in \NN^n$, we set
$\wt_{\mu}:=\prod_{i, \mu_i>0} x_i.$
Let $\mathcal{Q}_{\mu}^{\lambda}$ denote the set of (generalized)
two-line queues with bottom row $\mu$ and top row $\lambda$.
For $Q_0\in \mathcal{Q}_{\mu}^{\lambda}$, we define
\begin{align}
\wt(Q_0) &= \wt_{\pair}(Q_0) \cdot \wt_{\mu}\\
\label{amulambda}
a_{\mu}^{\lambda}& =
\sum_{Q_0\in \mathcal{Q}_{\mu}^{\lambda}} \wt_{\pair}(Q_0)\in \QQ(q,t)\\ 
\label{fmulambda}
f_{\mu}^{\lambda} & = f_{\mu}^{\lambda}(\xx;q,t) = \sum_{Q_0 \in \mathcal{Q}_{\mu}^{\lambda}} \wt(Q_0)=\wt_{\mu} \cdot a_{\mu}^{\lambda}.
\end{align}
\end{definition}

Note that the ``ball weight'' we associate to $Q_0$ only takes into account its bottom row.   This
is because we want $\wt(Q) = \wt(Q') \wt(Q_0)$, where the top $L-1$ rows of $Q$ give $Q'$ and the
bottom two rows give $Q_0$.

\begin{lem}\label{lem:recursive}\cite[Lemma 3.2]{CorteelMandelshtamWilliams2022}
We have the following recurrence for the homogeneous ASEP polynomials.
\[
f_{\mu}=\sum_\lambda f_{\mu}^{\lambda}f_{\lambda^-}.
\]
\end{lem}

It follows from the definitions that $f_{\mu}^{\lambda}$ is $0$ unless $\lambda$ has parts $0,2,3,..$
and is a permutation of the composition obtained from 
$\mu$ by replacing each part equal to $1$ by $0$.

\subsection{Generalized signed two-line queues}
In this section, we define a signed version of generalized two-line queues (\cref{sec:2line}), and we prove that the associated generating functions are encoded by the recurrence of \cref{def:b_recursion}.
\begin{definition}\label{def:ghost_two-row}
A \emph{generalized signed two-line queue} is a paired ball system obtained by considering the bottom two rows of a signed multiline queue. Its bottom row is represented by a composition $\mu\in \NN^n$, and its top row by a signed permutation $\alpha$ of $\mu$.
Let $\mcG^\alpha_\mu$ denote the
set of (generalized) signed two-line queues with bottom row $\mu$ and top row $\alpha$. 
\end{definition} See \Cref{fig:ghost_two_line_queue} for an example of a signed two-line queue. Using \eqref{eq:pair2}, we define the weight of a signed two-line queue $\G\in \mcG^\alpha_\mu$ to be
$$\wt_{\pair}(\G)=\prod_{p}\wt_{\pair}(p),$$ where the product is over all nontrivial (signed) pairings of $\G$. We then define the weight generating function $G_\mu^\alpha$ of $\mcG^\alpha_\mu$ to be $$G_\mu^\alpha=G_\mu^\alpha(t):=\sum_{\G\in\mcG_\mu^\alpha}\wt_{\pair}(\G).$$

\begin{figure}[t]
    \centering
    \includegraphics[width=0.5\linewidth]{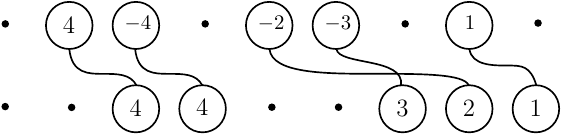}
    \caption{An example of a signed two-line queue in $\mcG_{(0,0,4,4,0,0,3,2,1)}^{(0,4,-4,0,-2,-3,0,1,0)}.$
    }
    \label{fig:ghost_two_line_queue}
\end{figure}

Recall that $\skipped$ is the statistic associated to a nontrivial pairing defined in \cref{def:wt}. We now give an equivalent ``static'' description of this statistic which follows directly from the definitions.
\begin{lem}\label{lem:skipped}
    Let $\G$ be a signed two-line queue and let $p$ be a nontrivial $a$-pairing connecting a signed ball labeled $\pm a$ in column $i$ of the top row, to a regular ball in column $j>i$ in the bottom row labeled $a$. Then $\skipped(p)$ counts the number of balls $B$ in the bottom row and in a column $r$ labeled by $c\in \NN_+$, such that $i<r<j$ and, either
    \begin{itemize}
        \item $c<a$,
        \item or $c=a$ and the ball to which $B$ is paired lies in a column $k<i$.
    \end{itemize}
\end{lem}

We now show below that the coefficients $(G^\alpha_\mu)$ satisfy the recursion of \cref{def:b_recursion}.

\begin{prop}\label{prop:recursion}
Fix $\mu\in\NN^n$ and $\alpha\in \ZZ^n$.  Then we have the following.
    \begin{enumerate}[font=\normalfont]
        \item\label{unpack:item 1} If $\mu\in\Pack(k,n)$ for some $0\leq k\leq n$, then
        $G^\alpha_{\mu}=\delta_{\mu,\lVert \alpha\rVert}.$
        \item\label{unpack:item 2} Given $1\leq i\leq n-1$ such that $\mu_i>0$ and $\mu_{i+1}=0$, we have
        \begin{equation}\label{eq:G}
        G^\alpha_{s_i \mu}=\sum_{\beta\in \NN^n}G^\beta_{\mu}\Ni_{\beta,\alpha} .
        \end{equation}
    \end{enumerate}
\end{prop}

For convenience, we rewrite  \eqref{eq:G} explicitly, by replacing the coefficients $\Ni_{\alpha, \beta}$ defined in \cref{def:matrix} by their values: the coefficient $G_{s_i\mu}^{\alpha}$ is equal to
\begin{subequations}\label{eq:G_recursion}
\begin{align}
    &t G^\alpha_\mu 
    &&\text{if $\alpha_i=\alpha_{i+1}$},\label{subeq:1}\\
    &       t G_\mu^{s_i\alpha}
    &&\text{if $\alpha_i>\alpha_{i+1}\geq 0$ or $-\alpha_i>-\alpha_{i+1}>0$},\label{subeq:2}\\
    &        G_\mu^{s_i\alpha}-(1-t)G^\alpha_{\mu}
    &&\text{if $\alpha_{i+1}>\alpha_{i}\geq 0$ or 
            $-\alpha_{i+1}>-\alpha_{i}> 0$},\label{subeq:3}\\
    &        G_\mu^{s_i\alpha}-(1-t)G_\mu^{\alpha}
    &&\text{if $-\alpha_i=\alpha_{i+1}>0$},\label{subeq:4}\\
    &         G_\mu^{s_i\alpha}-(1-t)G_\mu^{\alpha}
    &&\text{if $\alpha_{i+1}>-\alpha_{i}> 0$,}\label{subeq:5}\\
    &         tG_\mu^{s_i\alpha}-(1-t)G_\mu^{\alpha}
             &&\text{if $-\alpha_i>\alpha_{i+1}=0$,}\label{subeq:6}\\
    &         tG_\mu^{s_i\alpha}-(1-t)\left(G_\mu^{\alpha}-G^{\dots,-\alpha_i,-\alpha_{i+1},\dots}_\mu\right)
             &&\text{if $-\alpha_i>\alpha_{i+1}> 0$},\label{subeq:7}\\
    &         G_\mu^{s_i\alpha}
             &&\text{if $-\alpha_{i+1}>\alpha_{i}=0$,}\label{subeq:8}\\
    &         t G_\mu^{s_i\alpha}
    &&\text{if $\alpha_i>-\alpha_{i+1}> 0$},\label{subeq:9}\\
    &         G_\mu^{s_i\alpha}-(1-t)G^{\dots,-\alpha_i,-\alpha_{i+1},\dots}_\mu
             &&\text{if $-\alpha_{i+1}>\alpha_{i}>0$,}\label{subeq:10}\\
    &         t G_\mu^{s_i\alpha}&&\text{if $\alpha_i=-\alpha_{i+1}>0$}\label{subeq:11}.
\end{align}
\end{subequations}

\begin{proof}
We start by proving \cref{unpack:item 1} of the proposition. When $\mu\in\Pack(k,n)$ for some $0\leq k\leq n$, the balls in the bottom row of a signed two-line queue $\G\in\mcG_\mu^{\alpha}$  occupy positions $1,\dots,k$. Since in such a system all pairings go from left to right, this implies that the balls in the top row are also in positions $1,\dots,k$ and all pairings are trivial. Recall that a pairing connects two balls with labels of the same absolute value. This finishes the proof of 
\cref{unpack:item 1}.

    We now prove \cref{unpack:item 2}.  We will prove the most interesting of the cases (38a)--(38k); the reader can find the rest of the cases in \cite{BDWv1}.
    We start by introducing some notation. We will represent generating functions of signed multiline queues using diagrams. For example
    $$G_{\mu}^\beta= 
\left[\includegraphics[width=0.25\linewidth,valign=c]{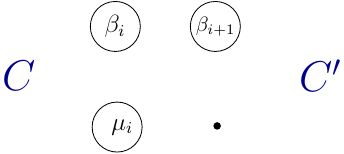}\right]$$
   represents a signed two-line queue 
   where $\mu_{i+1}=0$, and $C$ (respectively $C'$) is the part of the queue which lies in columns $j<i$ (respectively, $j>i+1$). 
So we want to prove that \begin{equation}\label{eq:diagrams}
\left[\includegraphics[width=0.25\linewidth,valign=c]{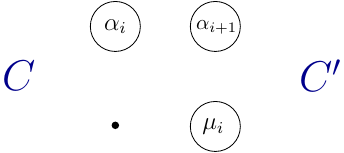}\right] =\sum_{\beta}\Ni_{\beta,\alpha} \left[\includegraphics[width=0.25\linewidth,valign=c]{Transition_diagrams/Transition01.pdf}\right].  
    \end{equation}
    To do so, we construct weight preserving bijection between these classes of multiline queues. In our bijections, the parts $C$ and $C'$ parts will not change,  and we will only be studying pairings which connect to at least one of the balls in columns $i$ and $i+1$. Thus, for simplicity, in the diagrams that follow, we will omit $C$ and $C'$.

    Recall that in each signed layer, the three configurations of \Cref{fig:forbidden_configurations_ghost} are forbidden. We will use without further mention that the contribution of diagrams containing one of these configurations is 0. This implies in particular that in the previous diagrams of \cref{eq:diagrams}, we always have $\alpha_i,\beta_{i+1}\leq 0$. In what follows, the (nonempty) balls will be represented by labels $\pm a,\pm b,\pm c$ with $a,b,c>0$. We will use the description of the $\skipped$ statistic given in \cref{lem:skipped}.

    The purpose of using the diagrammatic equations is that they are convenient to write decompositions of the generating functions. For example, we have
    \begin{align*}    \left[\includegraphics[height=0.12\linewidth,valign=c]{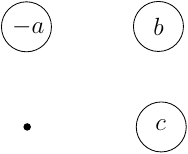}\right]=
 \left[\includegraphics[height=0.12\linewidth,valign=c]{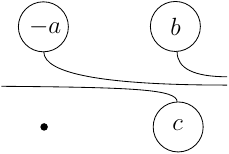}\right]
+\mathbbm{1}_{b=c}\left[\includegraphics[height=0.12\linewidth,valign=c]{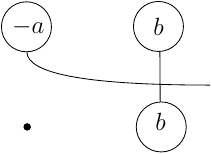}\right]+
\mathbbm{1}_{a=c}\left[\includegraphics[height=0.12\linewidth,valign=c]{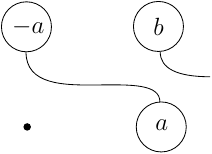}\right].
    \end{align*}
    
    \begin{itemize}
        \item  Case $\alpha_i=\alpha_{i+1}=0$ (First part of the proof of \cref{subeq:1}).  We claim that
        $$\left[\includegraphics[height=0.12\linewidth,valign=c]{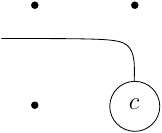}\right]=t\left[\includegraphics[height=0.12\linewidth,valign=c]{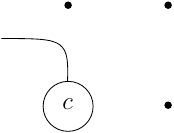}\right].$$
        This is easy to check: we get an extra factor of $t$ on the right hand side, because the number of empty positions contributing to the weight $\wt_{\pair}(p)$ of the pairing differs by 1.
        \item Case $\alpha_i=\alpha_{i+1}>0$ (Second part of the proof of \cref{subeq:1}).
        We want to prove that for $a,c>0$
        \begin{align*}
            \left[\includegraphics[height=0.12\linewidth,valign=c]{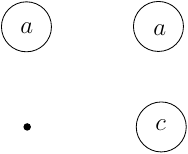}\right]
        &=t\left[\includegraphics[height=0.12\linewidth,valign=c]{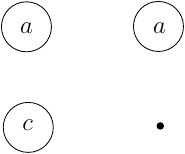}\right]
        \end{align*}
        This is trivially true since both of these diagrams contain forbidden configurations.
 \item Case $\alpha_i=\alpha_{i+1}<0$ (Third part of the proof of \cref{subeq:1}).
        We want to prove that for $a,c>0$
        \begin{align*}
\left[\includegraphics[height=0.12\linewidth,valign=c]{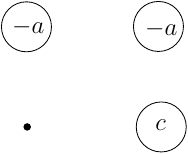}\right]
&=t\left[\includegraphics[height=0.12\linewidth,valign=c]{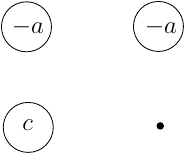}\right].
        \end{align*}
    First, we have
\begin{subequations}
\begin{align}\left[\includegraphics[height=0.12\linewidth,valign=c]{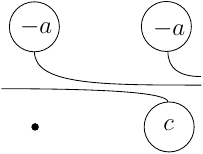}\right]
&=t\left[\includegraphics[height=0.12\linewidth,valign=c]{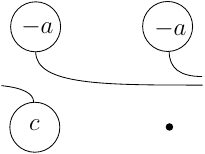}\right]\label{eq:transition11}
        \end{align}
        
        Note that in order to avoid the forbidden configurations, necessarily $c\leq a$.  Now the claim is true because the diagram on the left
        has one more skipped ball than the diagram on the right. Moreover,
        \begin{align}\left[\includegraphics[height=0.12\linewidth,valign=c]{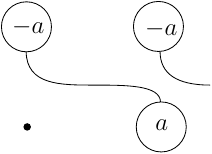}\right]
        &=-(1-t)\left[\includegraphics[height=0.12\linewidth,valign=c]{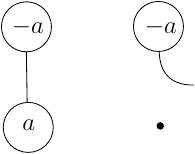}\right],
        \label{eq:transition12}
        \end{align}
        \begin{align}
\left[\includegraphics[height=0.12\linewidth,valign=c]{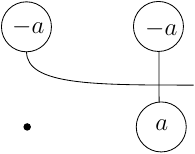}\right]
&=\left[\includegraphics[height=0.12\linewidth,valign=c]{Transition_diagrams/Transition82.pdf}\right].
    \label{eq:transition13}
        \end{align}
\end{subequations}
        The result is then obtained by summing $\text{\cref{eq:transition11}}+\mathbbm{1}_{c=a}\left(\text{\cref{eq:transition12}}+\text{\cref{eq:transition13}}\right)$.

    \item  Case $-\alpha_i>\alpha_{i+1}>0$ (Proof of \cref{subeq:7}). We want to prove that for $a>b>0$, we have
        \begin{align*}
\left[\includegraphics[height=0.12\linewidth,valign=c]{Transition_diagrams/Transition10.pdf}\right]
&=t\left[\includegraphics[height=0.12\linewidth,valign=c]{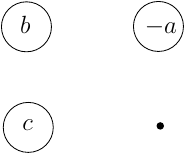}\right]-(1-t)\left[\includegraphics[height=0.12\linewidth,valign=c]{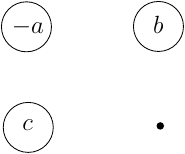}\right]+(1-t)\left[\includegraphics[height=0.12\linewidth,valign=c]{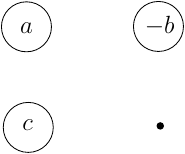}\right]\\
&=t\left[\includegraphics[height=0.12\linewidth,valign=c]{Transition_diagrams/Transition110.pdf}\right]+(1-t)\left[\includegraphics[height=0.12\linewidth,valign=c]{Transition_diagrams/Transition111.pdf}\right].
        \end{align*}
Notice that

\begin{subequations}
\begin{align}
\left[\includegraphics[height=0.12\linewidth,valign=c]{Transition_diagrams/Transition102.pdf}\right]
&=t\left[\includegraphics[height=0.12\linewidth,valign=c]{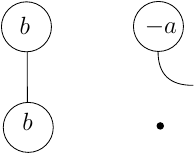}\right],\label{eq:transition21}
        \end{align}
        \begin{align}
\left[\includegraphics[height=0.12\linewidth,valign=c]{Transition_diagrams/Transition103.pdf}\right]
&=(1-t) \left[\includegraphics[height=0.12\linewidth,valign=c]{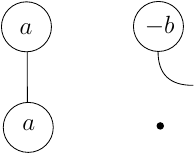}\right],\label{eq:transition22}
        \end{align}
here we multiply by $1-t$ since we go from a diagram with one nontrivial negative pairing, to a diagram with two nontrivial parings, one of them is positive and the other one is negative.
Finally, we prove that
\begin{align}
    \left[\includegraphics[height=0.12\linewidth,valign=c]{Transition_diagrams/Transition101.pdf}\right]=t\left[\includegraphics[height=0.12\linewidth,valign=c]{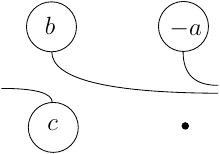}\right]+(1-t)\left[\includegraphics[height=0.12\linewidth,valign=c]{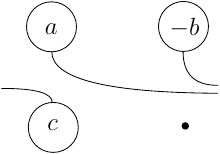}\right].\label{eq:transition23}
\end{align}
\end{subequations}

We distinguish two cases. When $b<a<c$ we have
        \begin{align*}
\left[\includegraphics[height=0.12\linewidth,valign=c]{Transition_diagrams/Transition101.pdf}\right]
&=\left[\includegraphics[height=0.12\linewidth,valign=c]{Transition_diagrams/Transition114.pdf}\right]\\
&=t\left[\includegraphics[height=0.12\linewidth,valign=c]{Transition_diagrams/Transition114.pdf}\right]+(1-t) \left[\includegraphics[height=0.12\linewidth,valign=c]{Transition_diagrams/Transition113.pdf}\right].
        \end{align*}
        where we used the fact that 
$$\left[\includegraphics[height=0.12\linewidth,valign=c]{Transition_diagrams/Transition114.pdf}\right]=\left[\includegraphics[height=0.12\linewidth,valign=c]{Transition_diagrams/Transition113.pdf}\right]$$
When $b<c\leq a$
\begin{align*}
    \left[\includegraphics[height=0.12\linewidth,valign=c]{Transition_diagrams/Transition101.pdf}\right]=t\left[\includegraphics[height=0.12\linewidth,valign=c]{Transition_diagrams/Transition114.pdf}\right]=t\left[\includegraphics[height=0.12\linewidth,valign=c]{Transition_diagrams/Transition114.pdf}\right]+(1-t)\left[\includegraphics[height=0.12\linewidth,valign=c]{Transition_diagrams/Transition113.pdf}\right]
\end{align*}
The result is obtained by summing 
$\mathbbm{1}_{c=b}\text{\cref{eq:transition21}}+\mathbbm{1}_{c=a}\text{\cref{eq:transition22}}+\text{\cref{eq:transition23}}$.
    \end{itemize}
For proofs of the remaining cases, see \cite{BDWv1}.
\qedhere
    
\end{proof}

\subsection{Completing the proof of the main theorem}
Recall that if $\mu = (\mu_1,\dots, \mu_n) \in \{0,1,\dots, L\}^n$ then $F_\mu^*(\bfx,q,t)$ is the generating function of the signed multiline queues of type $\mu$ (see \cref{def:Fmu}). Our goal is to prove that
$f_\mu^*=F_\mu^*$. We start with the following lemma.

 \begin{lem}\label{lem:F_decomposition}
 For any  composition $\mu$, we have
 \begin{equation}\label{eq:F_decomposition}
   F^*_\mu=\sum_{\lambda\in\NN^n}{F^{*\lambda}_\mu} q^{|\lambda^-|} F^*_{\lambda^{-}}(x_1/q,\dots,x_n/q),  
 \end{equation}
 where 
 $$F^{*\lambda}_{\mu}:=\sum_{\alpha\in\ZZ^n}G_\mu^\alpha \wt_\alpha a^\lambda_{\lVert \alpha\rVert}.$$
 \end{lem}
\begin{proof}
    A $(2L \times n)$ signed multiline queue $Q$ of type $\mu$  is obtained as follows:
    \begin{itemize}
        \item we choose a signed permutation $\alpha$ of $\mu$, and  a generalized signed two-line multiline queue $Q_0\in\mcG_\mu^{\alpha}$ (see \cref{def:ghost_two-row}),
        \item we choose a permutation $\lambda$ of the composition obtained from $\lVert \alpha\rVert$ by replacing $1$'s by $0$'s, and we choose a generalized two-line multiline queue $Q_1\in \mathcal{Q}_{\lVert \alpha\rVert}^\lambda$ (see \cref{def:two-row}),
        \item we choose a $(2(L-1) \times n)$ signed multiline queue $Q_2$ of type $\lambda^-$,
        \item we  glue $Q_1$ on  top of $Q_0$: in this operation, a ball $B_i$ from the top row of $Q_0$ labeled $\alpha_i$ is superposed with a ball $B'_i$ from the bottom row of $Q_1$ labeled $|\alpha_i|$, The new ball will then be labeled $\alpha_i$.
        \item we glue $Q_2$ on top of $Q_1$, after increasing the labels of all balls in $Q_2$ by 1.
    \end{itemize}
    Note that in this operation, the row of each ball in $Q_2$ increase by 1, and as a consequence the weight of each negative ball in $Q_2$ (as defined in \cref{def:wt}) is multiplied by $q$. Since $Q_2$ has $2|\lambda^{-}|$ balls (only half of them have weights),  the new shifted ball-weight  is obtained by
    $$\wt_{\ball}(Q)=\wt_{\alpha}q^{|\lambda^-|}\wt_{\ball}(Q_2)(x_1/q,\dots,x_n/q).$$
    Moreover, the pair weight of $Q$ is obtained as the product
    $$\wt_{\pair}(Q)=\wt_{\pair}(Q_0)\wt_{\pair}(Q_1)\wt_{\pair}(Q_2).$$
    We conclude using the fact that, by definition, $G_\mu^\alpha$ is the generating function of $\mcG_\mu^\alpha$ and $a_{\lVert \alpha\rVert}^\lambda$ is the generating function of $\mathcal{Q}_{\lVert \alpha\rVert}^\lambda$.
\end{proof}

In the following, we will use the convention that the empty signed multiline queue is the unique one of type $(0,\dots,0)$ and that it has total weight 1. As a consequence,
\begin{equation}\label{eq:F_empty}
    F^*_{(0,\dots,0)}=1,
\end{equation}
With this convention, \cref{lem:F_decomposition} holds in particular when $\mu = (\mu_1,\dots, \mu_n) \in \{0,1\}^n$. Indeed, in this case the only choice of $\lambda$ in \cref{eq:F_decomposition} is $(0,\dots,0)$, and the same proof then works. 

\begin{proof}[Proof of \cref{thm:main}]
    We proceed by induction on $L\geq 0$. When $L=0$, we have from \cref{eq:F_empty} that $F^*_{(0,\dots,0)}=1$. It is clear from the definitions (\cref{def:ASEP} and \cref{thm:Knop_Sahi}) that this corresponds to $f^*_{(0,\dots,0)}=E^*_{(0,\dots,0)}.$
    We now assume the result for all compositions $\lambda\in\{0,1,\dots, L\}^n$ and we fix $\mu\in\{0,1,\dots, L+1\}^n$.
    We start by applying \cref{lem:F_decomposition} and the induction assumption:
    $$F^*_\mu=\sum_{\lambda\in\NN^n}{F^{*\lambda}_\mu} q^{|\lambda^-|} F^*_{\lambda^{-}}\left(\frac{x_1}{q},\dots,\frac{x_n}{q}\right)=\sum_{\lambda\in\NN^n}{F^{*\lambda}_\mu} q^{|\lambda^-|} f^*_{\lambda^{-}}\left(\frac{x_1}{q},\dots,\frac{x_n}{q}\right),  
$$
with $F^{*\lambda}_{\mu}:=\sum_{\alpha\in\ZZ^n}G_\mu^\alpha \wt_\alpha a^\lambda_{\lVert \alpha\rVert}.$
We know from \cref{prop:recursion} that the coefficients $(G_\mu^\alpha)$ satisfy the recursion of \cref{def:b_recursion}. This allows us to apply  \cref{thm:f_decomposition}, and we get
\begin{equation}\label{eq:decomposition_f}
f^*_\mu=\sum_{\lambda}F^{*\lambda}_{\mu}(x_1,\dots,x_n)q^{|\lambda^-|}f^*_{\lambda^-}\left(\frac{x_1}{q},\dots,\frac{x_n}{q}\right)=F^{*}_{\mu}
\end{equation}
as desired.
\end{proof}

\begin{remark}
Using a variant of the combinatorial recursion given in \cref{lem:F_decomposition}, one can similarly show that the polynomial $h_\alpha$ defined by \cref{eq:def_h} is the weighted generating functions of a variant of signed multiline queues whose bottom row is Row 1', and that row has type $\alpha$. Indeed, such multiline queues are obtained by gluing generalized two-line queues as in the proof of \cref{lem:F_decomposition}, alternating signed and classical layers, but with the bottom layer being classical.     
\end{remark}

\section{A tableaux formula for interpolation Macdonald polynomials}\label{sec:tableau}

In this section we give a tableaux formula for interpolation ASEP and Macdonald
polynomials, see \cref{thm:tableaux}, and prove that it is equivalent to the signed multiline queue formula we gave in \cref{thm:main}.
We then give a tableaux formula for the 
\emph{integral form} $J_{\lambda}^*$ of interpolation Macdonald polynomials, see \cref{cor:integral_form}, and give a combinatorial proof of an integrality result, see \cref{thm:integrality}.

Let $\lambda=(\lambda_1,\dots,\lambda_n)$ be a partition with  $\lambda_i\in \NN$ and largest part $L$.
The \emph{(doubled) diagram} $\DD=\DD_{\lambda}$ associated to $\lambda$ is a
 sequence of $n$ columns of boxes where the $i$th column contains  $2\lambda_i$ boxes (justified to the bottom).
 We number the rows of $\DD$ from bottom to
top by $1, 1', 2, 2',\dots,L,L'$ and the columns from left to right (starting from column $1$). Abusing notation slightly,
we often use $\DD$ to refer to the collection of boxes in $\DD$. We let $\DD^r$ and $\DD^{r'}$ denote the collection of boxes in $\DD$ in row $r$ and $r'$, respectively.  We also let $\DD^{\classic}$ (respectively, $\DD^{\primed}$) denote
the set of boxes in $\DD$ that come from \emph{classic rows} $1,2,\dots,L$ (respectively, \emph{primed rows} $1',2',\dots,L'$).

We use $(i,j)$ to refer to the box in column $i$ and row $j$.
For a box $x=(i,j)$, we denote by $d(x)=(i,j^-)$ the box directly below it (if it exists).

\subsection{The tableaux formula for \texorpdfstring{$P_{\lambda}^*$}{P*}}

We now explain how to map each signed multiline queue to a tableau, 
in particular, to a filling of a diagram as above.

 \begin{definition}\label{def:Tab}
Suppose $\mu=(\mu_1,\dots,\mu_n)$ is a composition with maximal entry $L$
  and let $\G\in\SMLQ(\mu)$.
Let $\lambda$ be the partition obtained from $\mu$ by arranging its parts
in decreasing order.  
We define a total order on the strands of linked balls,
where the longest strands come earlier, and if two strands have the same length, the one whose top ball is to the right comes first.  
Now to each strand of linked balls we associate a column
whose entries record the column locations of its balls -- with a sign to indicate when a ball is signed -- 
and we then concatenate these columns according to the above total order.
Let $\Tab(\G)$ denote the resulting tableau.
        \end{definition}
It follows from the definition that the top entries of columns, when 
they are at the same height, are listed in decreasing order of their absolute value.

\begin{figure}[!ht]
\includegraphics[height=1.4in]{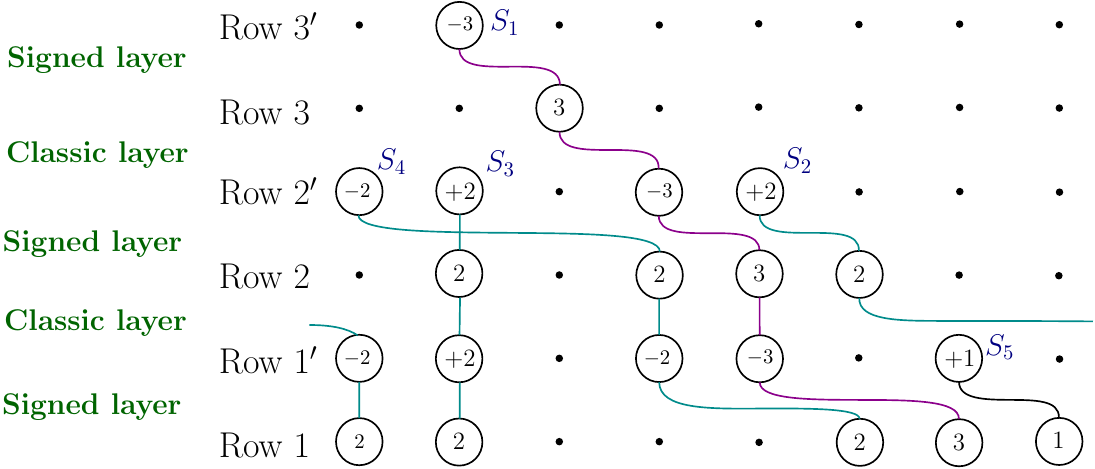} \hspace{.5in} \includegraphics[height=1.4in]{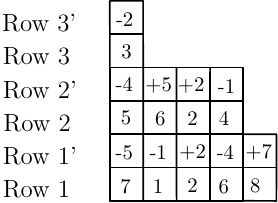}
 \caption{
At left: a signed multiline queue of type $(2,2,0,0,0,2,3,1)$.  The $i$-th strand (using the total order of \cref{def:Tab}) is labeled by $S_i$. At right: the corresponding signed queue tableau, where the $i$th column corresponds to the $i$th strand.
}
\label{fig:tableau_example}
 \end{figure}

\Cref{fig:tableau_example} illustrates
the signed multiline queue $\G$ from \Cref{fig:GMLQ_example} and the corresponding 
tableau $\Tab(\G).$
Our next goal is to characterize the tableau of the form $\Tab(\G)$, and rewrite our main theorem in terms of statistics on these tableaux.  We will define \emph{signed queue tableaux} in what follows; and as we define them, we will explain how their properties capture
the properties of signed multiline queues via the map $\Tab$ above.

\begin{definition}
        \label{def:filling}
        For $\lambda=(\lambda_1,\dots,\lambda_n)$ a partition, a \emph{filling} $\phi:D_{\lambda} \to [\pm n]$ of ${\DD}_{\lambda}$ is a map from $\DD_{\lambda}$ to $[\pm n]=\{1,2,\dots,n\} \cup \{-1,-2,\dots,-n\}$, such that:
\begin{itemize}
\item the top entries of columns, when they are at the same height, decrease in absolute value from left to right;
\item each classic row $r$ contains only positive integers, but a signed row $r'$ may contain both positive and negative integers;
\item if row $r'$ contains a positive integer $j$, then row $r$ must also contain
a $j$;\footnote{This requirement corresponds to the fact that in a signed row of a signed 
multiline queue $\G$, a regular ball cannot have an 
empty spot directly underneath it (see the rightmost forbidden configuration in \Cref{fig:forbidden_configurations_ghost})}.
\item we have that $|\phi(d(x))|\geq |\phi(x)|$ for any cell $x$ in a row $r'$.\footnote{This requirement corresponds to the fact that signed pairings cannot wrap around.}
\end{itemize}
We say that a box containing a positive integer (respectively, negative integer) is a \emph{positive cell} (respectively, \emph{negative cell}).
\end{definition}

Given any row $j$ of a diagram, we let $j^-$ denote the row directly under $j$, if it exists.  So we have that
$$j^-=\begin{cases} 
       r &\text{ if }j=r'\text{ for some $r$}\\
       (r-1)' &\text{ if }j=r\text{ for some $r\geq 2$}.
       \end{cases}
       $$

\begin{definition}\label{def:attacking}
Let $\phi:{\DD}_{\lambda} \to [\pm n]$ be a filling, 
 and let $(i,j)\in \DD_{\lambda}$.  If $(i,j)$ is a positive cell, then we say that it 
 \emph{attacks} the following boxes of ${\DD}_{\lambda}$:
\begin{itemize}
\item[(i.)] $(i',j)\in {\DD}_{\lambda}$ where $i\neq i'$,\footnote{\label{f:sameplace}This will correspond to the fact that in $\G$, we cannot have two balls in the same location.}
\item[(ii.)] $(i',j^-)\in {D}_{\lambda}$ where $i'\neq i$ such that $\lambda_i\geq \lambda_{i'}$,\footnote{\label{f:forbidden_1}This will correspond
to \Cref{fig:forbidden_configurations_classic} and the leftmost forbidden configuration of \Cref{fig:forbidden_configurations_ghost}.}
\end{itemize}
 If $(i,j)$ is a negative cell, then we say that it 
 \emph{attacks} the following boxes of $\DD_{\lambda}$:
\begin{itemize}
\item[(i.)] $(i',j)\in {\DD}_{\lambda}$ where $i\neq i'$,\footnoteref{f:sameplace}
\item[(ii.)] $(i',j^-) \in {\DD}_{\lambda}$ where $i'<i$ such that $\lambda_{i'}>\lambda_{i}$.\footnote{This will correspond to the middle forbidden configuration
of \Cref{fig:forbidden_configurations_ghost}.}
\end{itemize}
\end{definition}

\begin{definition} \label{def:QT}
        Let $\lambda=(\lambda_1,\dots,\lambda_n)$ be a partition.
         A \emph{signed queue tableau} of shape $\lambda$ is a filling $\phi:{\DD}_{\lambda}\to [\pm n]$ such that if one cell attacks another, the two 
         cells cannot contain entries with the same absolute value. 
We define the \emph{type} of the tableau to be the composition
$\mu=(\mu_1,\dots,\mu_n)$ such that $\mu_i$ equals half the height of the column which contains an $i$ in Row 1.
If $i$ does not occur in Row 1, $\mu_i=0$.
         Let $\T_{\lambda}^{\mu}$ denote the set of all signed queue tableaux of shape $\lambda$ and type $\mu$, and let 
          $\T_{\lambda}$ denote the set of all signed queue tableaux of shape $\lambda$.
\end{definition}

\begin{prop}
Choose a composition $\mu\in \NN^n$. The map $\Tab$ from \cref{def:Tab} gives a bijection between
the set $\SMLQ(\mu)$ of signed multiline queues of type $\mu$ and 
the set of signed queue tableaux $\T_{\lambda}^{\mu}$ of shape $\lambda$ and type $\mu$.    
\end{prop}
\begin{proof}
The proof is straightforward: the various properties of the definition of signed multiline queue get translated into properties of signed queue tableaux
as explained in the footnotes of \cref{def:filling} and \cref{def:attacking}.
\end{proof}

Our next goal is to translate the weight function on signed multiline queues to a weight function on signed queue tableaux.  First we need some notation.
Given a filling $\phi$ of ${\DD}_{\lambda}$, we say that a box $x$ is \emph{restricted} if the absolute values of the labels of $x$ and $d(x)$ are equal, i.e. $|\phi(d(x))|=|\phi(x)|$, and \emph{unrestricted} otherwise.
We make the convention that all boxes in row $1$ are restricted.

\begin{definition} \label{def:legarm}
Let $\lambda=(\lambda_1,\dots,\lambda_n)$ be a partition and let $\phi:{\DD}_{\lambda}\to [\pm n]$ be a signed queue tableau. 
Let $x=(i,j)$ be a box in a classic row.
We define $\leg(x)=\lambda_i - j$ to be the number of \emph{classic} boxes above $x$ in its column.
        The \emph{major index} is given by
        \[
        \maj(\phi)=\sum_{x\in \DD^{\classic}_{\lambda}\ :\ |\phi(d(x))|<\phi(x)} (\leg(x)+1).
        \]

        Given an unrestricted box $x=(i,j)$ of $\phi$, we define
        \begin{align}
                    \label{eq:arm} \arm(x)&=\#\big\{(k,j^-) \in D_{\lambda}\ :\ k>i,\ \lambda_{k}<\lambda_i \big\} \\
                &+ \#\big\{(k,j) \in D_{\lambda}\ :\ k>i,\ \lambda_{k}=\lambda_i,\mbox{and}\ (k,j) \text{ is unrestricted} \big\} \nonumber
        \end{align}
to be the number of boxes to the right of $x$ in the row below it, contained in columns shorter than its column, plus the number of unrestricted boxes to the right of and in the same row as $x$, contained in columns of the same length as $x$'s column.
\end{definition}
\begin{remark}\label{rem:legarm}
  The \emph{leg} statistic above will correspond to 
  the quantity $a-r$ in \eqref{eq:pair}.
\end{remark}
\begin{figure}[!ht]
  \centerline{\includegraphics[width=0.85\textwidth]{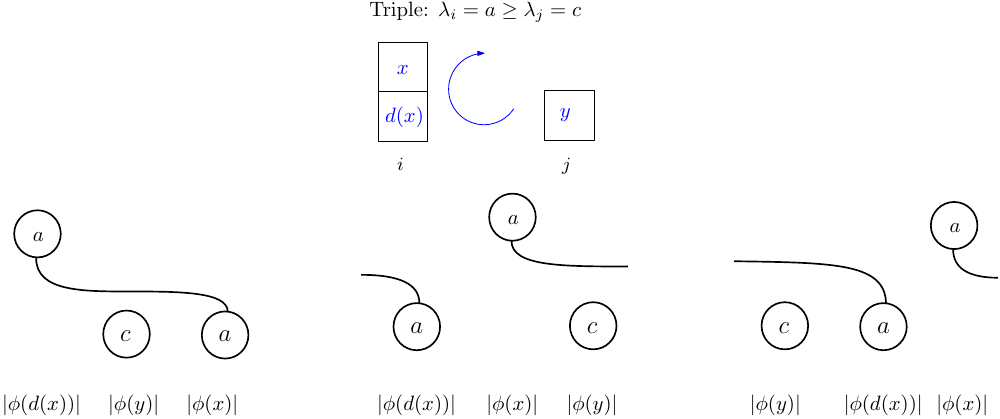}}
\centering
\caption{
A triple that forms  a coinversion, and the possibilities for the corresponding configuration in the signed multiline queue.   The arrow indicates the cyclic order of the labels. We  have $|\phi(d(x))|<|\phi(y)|<|\phi(x)|$, 
$|\phi(d(x))|<|\phi(x)|<|\phi(y)|$, and $|\phi(y)|<|\phi(d(x))|<|\phi(x)|$, respectively. Note that in the MLQ diagrams, the balls are represented with the absolute values of their labels, since these definitions do not depend on the sign.}
\label{fig:inversion}
 \end{figure}

        \begin{definition}\label{def:triple}
 A \emph{triple} is a triple of boxes $\{x, d(x), y\}$ in $\DD_{\lambda}$ where $x$ is in a classic or signed row,
  $y$ is to the right of and in the same row as $d(x)$, and either 
  \begin{enumerate}
      \item\label{item1} the column of $y$ is shorter than the column of $x$, or
      \item\label{item2} the column of $y$ has the same length as the column of $x$, and $u(y)$ (the cell just above $y$) is unrestricted.
  \end{enumerate}
  See \Cref{fig:inversion}. Notice that a triple implies that 
at the time that the balls labeled $a$ are paired, the ball labeled $c$ has not yet been paired to a ball in the row above.  Moreover
in \cref{item1} we have $c<a$, and in \cref{item2} the ball labeled $c$ is nontrivially paired.
        A {triple} is a \emph{coinversion}
        if $\phi(x)>0$, and either 
        $|\phi(x)|<|\phi(y)|<|\phi(d(x))|$, 
$|\phi(d(x))|<|\phi(x)|<|\phi(y)|$, or $|\phi(y)|<|\phi(d(x))|<|\phi(x)|$.
        {We then define $\coinv(\phi)$ to be the number of coinversions, as shown in \Cref{fig:inversion}.}  One may notice that if $(x,d(x),y)$ is a coinversion with $x\in D_\lambda^{\primed}$, then by item 4 of \cref{def:filling} we necessarily have $|\phi(x)|<|\phi(y)|<|\phi(d(x))|$. 
        
We define $\negative(\phi)$ to be the number of negative cells $x$ such that 
$|\phi(d(x))| \neq |\phi(x)|,$
 and  $\emp(\phi)$ to be the number of elements $0<a<b<c$ such that $\pm a$ appears in row $r'$, $c$ appears directly below $\pm a$ in row $r$, and $b$ does \emph{not} appear in row $r$. 
\end{definition}

\begin{definition}\label{def:weight}
Let $\lambda=(\lambda_1,\dots,\lambda_n)$ be a partition with largest part $L$, and let $\phi:{\DD}_{\lambda}\to [\pm n]$ be a signed queue tableau of shape $\lambda$. 
The \emph{weight} of $\phi$ is
\begin{equation}\label{eq:weight}
        \wt(\phi)=(-1)^{\negative(\phi)} q^{\maj(\phi)}t^{\coinv(\phi)+\emp(\phi)}\prod_{
        \substack{x\in D^{\classic}_{\lambda}\\  \text{$x$ unrestricted}}}         \frac{1-t}{1-q^{\leg(x)+1}t^{\arm(x)+1}} \prod_{\substack{x\in \DD^{\primed}_{\lambda}\\
        \text{$x$ unrestricted}}} (1-t),
\end{equation}

For a box $y\in {\DD}_{\lambda}^{r'}$ in row $r'$ of ${\DD}_{\lambda}$, 
we let 
$$\wt_{\phi}(y)=\begin{cases}
             x_{\phi(y)} & \text{ if }\phi(y)>0\\
             \frac{-q^{r-1}}{t^{n-1}} & \text{ if }\phi(y)<0.
             \end{cases}$$
We also define 
\begin{equation}
    x^\phi= \prod_{y\in D_{\lambda}^{\primed}} \wt_{\phi}(y)
    \end{equation}
to be
the Laurent monomial in $x_1,\ldots,x_n, q, t$ where the power of $x_i$ is the number of boxes in $D_{\lambda}^{\primed}$ whose entry is $i$, while the exponents of $q$ and $t$ depend on the number of negative entries in $D_{\lambda}^{\primed}$.
\end{definition}

We are now ready to state our tableaux version of \cref{thm:main}.

\begin{thm}\label{thm:tableaux}
Let $\lambda=(\lambda_1,\dots,\lambda_n)$ be a partition, and
let $\mu\in S_n(\lambda)$ be a composition. Then the interpolation ASEP polynomial $f_{\mu}^*(\bfx;q,t)$ equals the 
weight-generating function for signed queue tableaux $\T_{\lambda}^{\mu}$, that is,
\begin{equation}
    f_{\mu}^*(\bfx;q,t)= \sum_{\phi\in\QT_{\lambda}^{\mu}} \wt(\phi)x^\phi.
\end{equation}
And the interpolation Macdonald polynomial $P^*_{\lambda}(\xx; q, t)$
is equal to the weight-generating function for all signed 
queue tableaux $\QT_{\lambda}$ of shape $\lambda$, that is, 
$$P^*_{\lambda}(\xx;q,t) = \sum_{\phi\in\QT_{\lambda}} \wt(\phi)x^\phi.$$ 
\end{thm}

In order to prove the theorem, 
we will actually use a different convention for the ordering of pairings in our multiline queue and hence slightly different versions of the skipped and free statistics (the empty statistic does not depend on the pairing order).

\begin{definition}\label{def:new_order}[New pairing order]
We define the following new pairing order for multiline queues: for each row $r$ (classic or primed) we read the balls in decreasing order of the absolute value of their label; within a fixed label, we start by making the trivial pairings, and then pair balls with respect to the order of their strands given in \cref{def:Tab}. 

As in \cref{def:wt0} and \cref{def:wt}, we define the statistic $\free'$ and $\skipped'$ relative to this new order: let $p$ be a pairing from row $r$ to row $r^-$. Then
 $\free'(p)$ counts the number of balls in row $r^-$ that have not yet been matched right before we place the pairing $p$. Similarly, if the pairing $p$ matches a ball labeled $a$ in row $R$ and column $j$ to a ball in row $r^-$ and column $j'$, then the statistic $\skipped'(p)$ counts the number of free balls in row $r^-$ and columns $j+1, j+2,\dots, j'-1$  (indices considered modulo $n$). 

This gives rise to a new weight $\wt_{\pair}'(p)$ defined as in \cref{eq:pair} (respectively \cref{eq:pair2}) when $p$ is in classic layer (respectively signed layer). 

This defines a weight $\wt'(Q^\pm)$ for any signed multiline queues $Q^\pm$.
\end{definition}

It turns out that the weighted generating function of signed multiline queues is invariant under changes in the order in which non-trivial pairings within the same label are made.
This was proven for classical layers in \cite[Lemma 2.1]{CorteelMandelshtamWilliams2022} by constructing an involution that switches the order of two non-trivial pairings of the same label. The same argument applies to signed layers\footnote{One starts by noticing that in signed layers, the statistic $\skipped$ does not change if we first make the trivial pairings, and then we make the other pairings from right to left.}. We leave the details to the reader.

Since, in each layer, the orders of \cref{def:wt} and \cref{def:new_order}  differ only on non-trivial pairings of the same label, we get the following lemma.
\begin{lem}
    For any composition $\mu$, we have
    $$F^*_\mu(\bfx;q,t)=\sum_{Q\in\SMLQ_\mu}\wt(Q^\pm).$$
\end{lem}

\begin{lem} \label{rem:skipped}
 Under the map $\Tab$, the coinversion statistic  corresponds to the \emph{$\skipped'$} statistic 
 and the  $\arm$ statistic corresponds to the $\free'$ statistic
 for  signed multiline queues.
 \end{lem}
 \begin{proof}
 Consider a pairing between balls labeled $a$ which skips over a ball which will (eventually) be labeled by $c$,
 and let $x$, $d(x)$, and $y$ denote the cells of the tableau which correspond to the two paired balls labeled $a$ and the skipped ball labeled $c$.  Then $c \leq a$.  Since we use the pairing order from \cref{def:new_order}, the string of linked balls containing $c$ gives rise to a column $j$ in the corresponding tableaux which is to the right of the column $i$ containing $x$ and $d(x)$.
 Moreover
 $\lambda_i = a \geq c = \lambda_j$, and $y$ is not part of a trivial pairing, so $\{x, d(x), y\}$  form a triple.  The condition that the ball labeled $c$ is skipped by the pair exactly corresponds to the cyclic order given in the definition of coinversion, see \Cref{fig:inversion}. \end{proof}

\begin{remark}
 In \cite{CorteelMandelshtamWilliams2022}, the coinversion statistic on tableaux was computed by counting both ``Type A quadruples'' and ``Type B triples.''  However, by working with the pairing order from \cref{def:new_order}, we can work with (Type B) triples only.  We thank Olya Mandelshtam for explaining this to us; see also \cite{MandelshtamTBD}.   
\end{remark}

\begin{proof}[Proof of \cref{thm:tableaux}]
It is not hard to see that under the bijection $\Tab$, the statistics on signed multiline queues translate into corresponding statistics on signed queue tableaux, see
\cref{rem:legarm} and \cref{rem:skipped}.  Moreover, the empty statistic from \eqref{eq:pair2} corresponds to the empty statistic $\emp(\phi)$ on tableaux, while the factors of $-1$ in \eqref{eq:pair2} correspond to the statistic $\negative(\phi)$ on tableaux.  The product in 
\eqref{eq:weight} corresponds to a product over all nontrivial pairings.
\end{proof}

\subsection{The tableaux formula for the integral form}\label{ssec:integral_form}
In this section we will give a tableau formula for the integral normalization of the interpolation symmetric Macdonald polynomials and the interpolation ASEP polynomials. We start with some definitions.

Fix a partition $\lambda$ and a filling $\phi$ of its doubled diagram $D_\lambda$. Fix a signed cell $x\in D_\lambda^{\classic}$. We recall that $\arm(x)$ was defined in \cref{def:legarm}. We now define this statistic for signed cells. If $x=(i,j')$ is a signed cell, we will denote $u(x):=(i,j+1)$  the classic cell on top of $x$. We will use the convention that if this cell is not in the diagram $D_\lambda$, then it is restricted. 

As for classic cells, the $\leg$ of a signed cell will be defined as the number of classic boxes above $x$ in its column. In particular, for $x\in D_\lambda^{\primed}$ we have $\leg(x)=\leg(u(x))+1$.

We now extend the definition of $\arm$ (see \eqref{eq:arm}) to signed cells $x\in \DD_\lambda^{\primed}$,  in such a way that if $u(x)\in \DD_\lambda$, $u(x)$ is unrestricted, and $p$ is the pairing connecting the balls corresponding to the cells $x$ and $u(x)$ under the bijection $\Tab$, then $\arm(x)=\free(p)$. 
Even though our previous tableaux formula used the arm and leg statistics only for classic cells, our next formula can be written more cleanly if we shift these statistics to signed cells.

\begin{definition}
 Let $x\in \DD_\lambda^{\primed}$ be a signed cell.

If $u(x)$ is unrestricted,  we define 
                $\arm(x):=\arm(u(x))$, and if $u(x)$ is restricted, we define
\begin{align*}
\arm(x):&=\#\big\{(k,j) \in D_{\lambda}\ :\  k>i\big\}\\
                &+ \#\big\{(k,j) \in D_{\lambda}\ :\ k<i,\ \lambda_{k}=\lambda_i,\mbox{and}\ u(k,j) \text{ is unrestricted}\big\}.
\end{align*}
\end{definition}
\Cref{fig:Leg_arm} shows the statistics $\leg$ and $\arm$ for the signed multiline tableau of \Cref{fig:tableau_example}.
\begin{figure}[t]
    \centering
    \includegraphics[height=1.4in]{GhostTableaux2.pdf}\hspace{0.5cm}
    \includegraphics[height=1.4in]{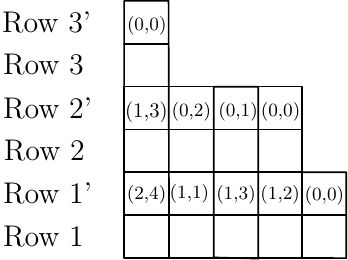}
    \caption{On the left: the tableau of \Cref{fig:tableau_example}. On the right: the pair $(\leg,\arm)$ for each  signed cell of the tableau.
    }
    \label{fig:Leg_arm}
\end{figure}

  Since $\leg(u(x))+1=\leg(x),$
the weight $\wt(\phi)$ (see~\cref{def:weight}), can be written as:
\begin{equation*}
        \wt(\phi)=(-1)^{\negative(\phi)} q^{\maj(\phi)}t^{\coinv(\phi)+\emp(\phi)}
        \prod_{\substack{x\in D^{\primed}_{\lambda}\\
        \text{$u(x)$ unrestricted}}}
        \frac{1-t}{1-q^{\leg(x)}t^{\arm(x)+1}}
        \prod_{\substack{x\in \DD^{\primed}_{\lambda}\\
        \text{$x$ unrestricted}}} (1-t).
\end{equation*}

We define $$\hook_\lambda:=\prod_{x\in D_\lambda^{\primed}}(1-q^{\leg(x)}t^{\arm(x)+1}).$$
One can show that unlike the definition of $\arm$, the definition of $\hook_\lambda$ is independent of the filling $\phi$, and corresponds to the usual hook product; see \cite[Section 5]{Knop1997b}. 

\begin{definition}
We define the \emph{integral (form) interpolation Macdonald polynomial} 
and the \emph{integral (form) interpolation ASEP polynomial} to be
$$J_\lambda^*:=\hook_\lambda P^*_\lambda \hspace{1cm} \text{ and }\hspace{1cm}
\hook_\lambda f_{\mu}^*.$$
\end{definition}

We define the \emph{integral weight} to be $\wt^J(\phi)
        :=\hook_\lambda \wt(\phi)$, which equals
\begin{equation*}
        (-1)^{\negative(\phi)} q^{\maj(\phi)}t^{\coinv(\phi)+\emp(\phi)}\hspace{-0.7cm}\prod_{\substack{x\in D^{\primed}_{\lambda}\\ \text{$u(x)$ unrestricted}}}\hspace{-0.3cm}
        (1-t) 
        \prod_{\substack{x\in \DD^{\primed}_{\lambda}\\ \text{$x$ unrestricted}}}\hspace{-0.25cm} (1-t)
        \prod_{\substack{x\in D^{\primed}_{\lambda}\\ \text{$u(x)$ restricted}}}\hspace{-0.4cm}
        (1-q^{\leg(x)}t^{\arm(x)+1}).
\end{equation*}

We then get the following corollary of \cref{thm:tableaux}.
\begin{cor}\label{cor:integral_form}
    Let $\lambda=(\lambda_1,\dots,\lambda_n)$ be a partition,
and let $\mu\in S_n(\lambda)$ be a composition. 
Then the integral form interpolation ASEP polynomial equals the generating function for signed queue tableaux  $\QT_{\lambda}^{\mu}$ counted with integral weights, that is,
\begin{equation}
    \hook_\lambda f_{\mu}^*= \sum_{\phi\in\QT_{\lambda}^{\mu}} \wt^J(\phi)x^\phi.
\end{equation}
And the integral interpolation Macdonald polynomial $J^*_{\lambda}(\xx; q, t)$
equals  the generating function for all signed 
queue tableaux $\QT_{\lambda}$ of shape $\lambda$ counted with integral weights, that is, 
$$J^*_{\lambda}(\xx;q,t) = \sum_{\phi\in\QT_{\lambda}} \wt^J(\phi)x^\phi.$$ 
\end{cor}

We deduce from these combinatorial formulas the following integrality results.
\begin{thm}\label{thm:integrality}
Fix a partition $\lambda\in \YY_n$.
    Consider the expansions of $J_\lambda^*$ in the monomial basis
    $J^*_\lambda=\sum_{\nu\in \YY_n:|\nu|\leq |\lambda|} c_{\lambda,\nu}\ m_\nu$. Then $t^{(n-1)(|\lambda|-|\nu|)}c_{\lambda,\nu}\in\ZZ[q,t].$
    Similarly, let $\mu\in S_n(\lambda)$ be a permutation of $\lambda$, and consider the expansion
    $\hook_\lambda \ f^*_\mu=\sum_{\nu\in \NN_n:|\nu|\leq |\mu|} d_{\mu,\nu} x^\nu$. Then $t^{(n-1)(|\mu|-|\nu|)}d_{\mu,\nu}\in\ZZ[q,t].$
\end{thm}
The first part of this theorem was obtained in \cite[Corollary 5.5]{Knop1997b} (see also~\cite[Theorem 5.3]{Sahi1996}). The second part is however new.
\begin{proof}
    In the combinatorial formulas given in \cref{cor:integral_form}, the weights are polynomials in the variables $x_i$ with coefficients in $\ZZ[q,t]$, except the weights assigned for negative boxes $y\in {D}_\lambda^{r'}$ for which $\wt_\phi(y)=\frac{-q^{r-1}}{t^{n-1}}$. Notice that the total number of boxes in ${D}_\lambda^{\primed}$ corresponds to $|\lambda|$, and that extracting a monomial $m_\nu$ in $J_\lambda^*$ corresponds to considering tableaux with $|\nu|$ positive boxes. As a consequence, $|\lambda|-|\nu|$ is the number of negative boxes in such a tableau, hence, by multiplying by $t^{(n-1)(|\lambda|-|\nu|)}c_{\lambda,\nu}$ we compensate all the denominators. The same reasoning applies to $\hook_\lambda f^*_\mu$. 
\end{proof}

\section{Application: factorization  of interpolation Macdonald polynomials}\label{sec:factorization}

Fix $n\geq 1$. Let ${\llbracket n \rrbracket \choose k}$ denote the $k$-element subsets of $\llbracket n \rrbracket$. For any $k\geq 0$, we define
$$e^*_{k}(x_1,\dots,x_n;t):=\sum_{ S\in {\llbracket n \rrbracket \choose k}} 
\prod_{i\in S}\left(x_i-\frac{t^{\#S^c\cap \llbracket i-1\rrbracket}}{t^{n-1}}\right),$$
where $S^c$ denotes the complement of $S$ in $\llbracket n\rrbracket$.
The top homogeneous part of $e_k^*$ is the $k$-th elementary symmetric function $e_k$. Even though it is not completely clear from the definition, the functions $e^*_k$ are symmetric (see \cref{eq:P_column} below).

The purpose of this section is to prove the following factorization formula for interpolation nonsymmetric Macdonald polynomials specialized at $q=1$, after a ``partial symmetrization''.
\begin{thm}\label{thm:factorization}
Let $\lambda$ be a partition with $\lambda_1>0$. For any subset 
$S \in {\llbracket n\rrbracket \choose \ell(\lambda)}$, we have
\begin{equation}\label{eq:partial_symmetrization}
  \sum_{\mu\in S_n(\lambda):\Supp(\mu)=S} f^*_\mu(x_1,\dots,x_n;1,t)=\prod_{i\in S}\left(x_i-\frac{t^{\#S^c\cap \llbracket i-1\rrbracket}}{t^{n-1}}\right)\prod_{2\leq i\leq \lambda_1}e^*_{\lambda'_i}(x_1,\dots,x_n;t).
\end{equation}
As a consequence, we have that for any partition $\lambda$
    \begin{equation}\label{eq:factorization}
      P^*_\lambda(x_1,\dots,x_n;1,t)=\prod_{1\leq i\leq \lambda_1}P^*_{\lambda'_i}(x_1,\dots,x_n;1,t)=\prod_{1\leq i\leq \lambda_1}e^*_{\lambda'_i}(x_1,\dots,x_n;t),
    \end{equation}
    where $\lambda'$ is the partition conjugate to $\lambda$.
\end{thm}
The symmetric part of this theorem (\cref{eq:factorization}) was proved in more generality in \cite{Dolega2017}.
Notice also that the top homogeneous part of \cref{eq:factorization} corresponds to the factorization property of (homogeneous) Macdonald polynomials; see~\cite[Chapter VI, Eq. (4.14.vi)]{Macdonald1995}. The nonsymmetric part of the theorem (\cref{eq:partial_symmetrization}) seems however to be new.

We start by giving a formula for interpolation ASEP polynomials indexed by $\mu\in\{0,1\}^n$ (for general $q$ and $t$). 
We will use the natural bijection between compositions $\mu\in\{0,1\}^n$ and subsets of $\llbracket n\rrbracket$ given by
    $\mu\mapsto S_\mu:=\{i:\mu_i=1\}$.
\begin{lem}\label{lem:factorization_one_row}
        For any $\mu\in\{0,1\}^n$, we have
        \begin{equation}\label{eq:factorization_one_row}
          f_\mu^*(x_1,\dots,x_n;q,t)=\prod_{i\in S_\mu}\left(x_i-\frac{t^{\#S^c_\mu\cap \llbracket i-1\rrbracket}}{t^{n-1}}\right). \end{equation}
        As a consequence, 
        \begin{equation}\label{eq:P_column}
          P^*_{(1^k,0^{n-k})}(x_1,\dots,x_n;q,t)=e^*_k(x_1,\dots,x_n;t).  
        \end{equation}
\end{lem}

\begin{proof}
We use \cref{thm:characterization}.
For $\mu\in \{0,1\}^n$, let $g_{\mu}$ denote the right-hand side of \eqref{eq:factorization_one_row}.
It is clear from the formula for $g_{\mu}$ that the second condition of \cref{thm:characterization} holds.
To show that the first condition of \cref{thm:characterization} holds, 
consider any composition $\nu$ such that 
$|\nu|\leq |\mu|$
and $\nu\notin S_n(\lambda)$.   
Note that if $\nu_i=0$, then 
\begin{equation*}
    k_i(\nu) = \#(S_{\nu} \cap [i-1]) + (n-i) =i-1-\#(S_{\nu}^c \cap [i-1]) +(n-i) = (n-1)-\#(S_{\nu}^c \cap [i-1]),
\end{equation*}
where $k_i(\nu)$ is the statistic defined in \cref{eq:k}. 

We claim that there exists an $i$ such that $\nu_i=0$ and $\#(S_{\nu} \cap [i-1])=\#(S_{\mu} \cap [i-1])$.  If we know the claim, then 
$k_i(\nu) = (n-1)-\#(S_{\mu}^c \cap [i-1])$, so 
 $g_{\mu}(\tilde{\nu})$ is obtained by plugging in $x_i=q^{\nu_i} t^{-k_i(\nu)} = t^{\#(S_{\mu}^c \cap [i-1])-(n-1)}.$
Thus $g_{\mu}(\tilde{\nu})=0$, and the uniqueness property of \cref{thm:characterization} implies that 
$f_{\mu}^* = g_{\mu}$. 

To prove the claim, consider the function $$\phi:i\in[1,n+1]\mapsto \#(S_{\nu}^c \cap [i-1])-\#(S_{\mu}^c \cap [i-1]) \in \ZZ.$$
This function has the property that $|\phi(i+1)-\phi(i)|\in \{0,1\}$, $\phi(1)=0$, and 
(because $\nu$ has more $0$'s than $\mu$) $\phi(n+1)<0$. 
Thus  we can find $i$ such that $\phi(i)=0$ and $\phi(i+1)=-1$, which implies that $\nu_i=0$ and the claim.

Now \eqref{eq:P_column} is obtained by summing \eqref{eq:factorization_one_row} over all $\mu$ of size $k$ (see \cref{prop:symmetrization}).
\end{proof}

Given $\mu$, let $\Supp(\mu):=\{i:\mu_i>0\}.$
 The following lemma is implicit in the discussion around \cite[(5.1)-(5.3)]{AMW} when $\lambda$ has distinct parts; we give a quick sketch below.

\begin{lem}\label{lem:sum_a_q1}
Fix a partition $\lambda$ with largest part $L$, and $\nu\in S_n(\lambda\backslash 1^{m_1(\lambda)})$, where $\lambda\backslash 1^{m_1(\lambda)}$ is the partition obtained from $\lambda$ by removing parts of size 1. We also fix a set $S \in {\llbracket n\rrbracket \choose \ell(\lambda)}$.
Then 
    $$\sum_{\mu\in S_n(\lambda):\Supp(\mu)= S}a^\nu_{\mu}(q=1)=1.$$
\end{lem}
\begin{proof}
Recall from \cref{sec:2line} that the coefficients $a^\nu_\mu$ enumerate generalized 2-line queues in $\QT^\nu_\mu$ according to their pairing weights. By fixing $\nu$, we fix the labels of balls in the top row. The positions but not the labels of the balls in the bottom row are fixed by $S$.  Such a multiline queue is obtained as follows: we start with the highest label $L$ in the top row, and trivially pair any ball having a ball directly underneath it. We then pair the rightmost free ball $B$ labeled $L$ (in the top row). If there are $r$ free balls left, then $B$ will have $r$ pairing choices, with weights $\frac{1-t}{1-t^r}$,$\frac{1-t}{1-t^r}t,\dots,\frac{1-t}{1-t^r}t^{r-1}$ (recall that  $q=1$). 
Thus the total weight of all possible pairings for $B$ is 1.  We then move on to the other balls in the top row (always choosing the rightmost ball with the largest label).
Note that  $\#S=\ell(\lambda)\geq\ell(\nu)$, this guarantees that all balls in the top row are paired: this fixes the labels of the paired balls in the bottom row, the unpaired ones will be labeled by 1.
In conclusion, when $q=1$, the total weight at each step is 1, and the lemma follows.
\end{proof}

\begin{proof}[Proof of \cref{thm:factorization}]
    We prove the result by induction on the size of the first part of $\lambda$. When $\lambda_1=0$, we have 
    $$P^*_{0^n}=f^*_{0^n}=1,$$
    which corresponds in \cref{eq:factorization} to an empty product.
The result was also proven for $\lambda_1=1$ in \cref{lem:factorization_one_row}.

Now fix $\lambda$ with $\lambda_1>0$. We start by showing \cref{eq:partial_symmetrization}, i.e that for any subset 
$S \in {\llbracket n\rrbracket \choose \ell(\lambda)}$ we have
\begin{equation*}
  \sum_{\mu\in S_n(\lambda):\Supp(\mu)=S} f^*_\mu(x_1,\dots,x_n;1,t)=\prod_{i\in S}\left(x_i-\frac{t^{\#S^c\cap \llbracket i-1\rrbracket}}{t^{n-1}}\right)\prod_{2\leq i\leq \lambda_1}e^*_{\lambda'_i}(x_1,\dots,x_n;t).
\end{equation*}
This equation will be proved by induction on $\sum_{i\in S}i$. Our base case is  the ``packed subset'' $S:=\{1,\dots,\ell(\lambda)\}$. In this case, we know from \cref{thm:recurrence} that
\begin{equation*}
  \sum_{\substack{\mu\in S_n(\lambda)\\
  \Supp(\mu)=S}} f^*_\mu(x_1,\dots,x_n;1,t)
= \prod_{1\leq i\leq \ell(\lambda)}\left(x_i-\frac{1}{t^{n-1}}\right)
\sum_{\nu}
f^*_{\nu^-}(x_1,\dots,x_n;1,t)\sum_{\substack{\mu\in S_n(\lambda)\\
\Supp(\mu)=S}}a^\nu_{\mu}(q=1),
\end{equation*}
where the first sum in the right-hand side is taken over compositions $\nu\in S_n(\lambda\backslash 1^{m_1(\lambda)})$.  
We now apply \cref{lem:sum_a_q1} and \cref{prop:symmetrization}, obtaining
\begin{align*}
  \sum_{\mu\in S_n(\lambda):\Supp(\mu)=S}
  f^*_\mu(x_1,\dots,x_n;1,t)
  &=\prod_{1\leq i\leq \ell(\lambda)}\left(x_i-\frac{1}{t^{n-1}}\right)
\sum_{\nu\in S_n(\lambda\backslash 1^{m_1(\lambda)})}
f^*_{\nu^-}(x_1,\dots,x_n;1,t)\\
  &=\prod_{1\leq i\leq \ell(\lambda)}\left(x_i-\frac{1}{t^{n-1}}\right)
\sum_{\nu\in S_n(\lambda^-)}
f^*_{\nu}(x_1,\dots,x_n;1,t)\\
&=\prod_{1\leq i\leq \ell(\lambda)}\left(x_i-\frac{1}{t^{n-1}}\right)
P^*_{\lambda^-}(x_1,\dots,x_n;1,t).
\end{align*}
Applying the induction hypothesis of \cref{eq:factorization}  with $\lambda^{-}$ (because $(\lambda^-)_1=\lambda_1-1<\lambda_1$), 
we get
$$\sum_{\mu\in S_n(\lambda):\Supp(\mu)=S}f^*_\mu(x_1,\dots,x_n;1,t)=\prod_{1\leq i\leq \ell(\lambda)}\left(x_i-\frac{1}{t^{n-1}}\right)\prod_{2\leq i\leq \lambda_1}e^*_{\lambda'_i}(x_1,\dots,x_n;t).$$
This finishes the proof of the base case \cref{eq:partial_symmetrization}. 

We now fix $S$ such that \cref{eq:partial_symmetrization} holds, and let $1\leq i\leq n-1$ such that $i\in S$ but $i+1\notin S$, and let $S':=S\backslash\{i\}\cup\{i+1\}.$ We want to prove the result for $S'$. Note that the action of the transposition $s_i$ is a bijection between $\{\mu\in S_n(\lambda):\Supp(\mu)=S\}$ and $\{\kappa\in S_n(\lambda):\Supp(\kappa)=S'\}$ and for any $\mu$ in the first set, we have
$f^*_{s_i\mu}=T_if^*_{\mu}$ (see \cref{prop:T_i-f_star}). Hence,
\begin{equation*}
  \sum_{\kappa\in S_n(\lambda):\Supp(\kappa)=S'}f^*_\kappa(x_1,\dots,x_n;1,t)
  =T_i\left(\prod_{j\in S}\left(x_j-\frac{t^{\#S^c\cap \llbracket j-1\rrbracket}}{t^{n-1}}\right)\prod_{2\leq j\leq \lambda_1}e^*_{\lambda'_j}(x_1,\dots,x_n;t)\right). 
\end{equation*}
Since the functions $e^*_{\lambda'_i}$ are symmetric, we obtain 
\begin{align*}
  \sum_{\kappa\in S_n(\lambda):\Supp(\kappa)=S'}f^*_\kappa(x_1,\dots,x_n;1,t)
  &=\prod_{2\leq j\leq \lambda_1}e^*_{\lambda'_j}(x_1,\dots,x_n;t)T_i\left(\prod_{j\in S}\left(x_j-\frac{t^{\#S^c\cap \llbracket j-1\rrbracket}}{t^{n-1}}\right)\right)\\
&=\prod_{2\leq j\leq \lambda_1}e^*_{\lambda'_j}(x_1,\dots,x_n;t)\prod_{j\in S'}\left(x_j-\frac{t^{\#(S')^c\cap \llbracket j-1\rrbracket}}{t^{n-1}}\right). 
\end{align*}
To obtain the last line, we use the fact that $T_i x_i=x_{i+1}$, and $T_i t^{\#S^c\cap \llbracket i-1\rrbracket}=t\cdot  t^{\#S^c\cap \llbracket i-1\rrbracket}=t^{\#(S')^c\cap \llbracket i\rrbracket}$. 

This finishes the proof of \cref{eq:partial_symmetrization}. We now sum \cref{eq:partial_symmetrization} over all subsets  $S \in {\llbracket n\rrbracket \choose \ell(\lambda)}$, getting
\begin{equation*}
  \sum_{\mu\in S_n(\lambda)} f^*_\mu(x_1,\dots,x_n;1,t)=\sum_{S:\#S=\ell(\lambda)}\prod_{i\in S}\left(x_i-\frac{t^{\#S^c\cap \llbracket i-1\rrbracket}}{t^{n-1}}\right)\prod_{2\leq i\leq \lambda_1}e^*_{\lambda'_i}(x_1,\dots,x_n;t),
\end{equation*}
which by \cref{prop:symmetrization} is equivalent to 
\begin{align*}
  P^*_\mu(x_1,\dots,x_n;1,t)
  &=e^*_{\ell(\lambda)}(x_1,\dots,x_n;t)\prod_{2\leq i\leq \lambda_1}e^*_{\lambda'_i}(x_1,\dots,x_n;t)\\
  &=\prod_{1\leq i\leq \lambda_1}e^*_{\lambda'_i}(x_1,\dots,x_n;t),
\end{align*}
which gives \cref{eq:factorization}. This finishes the proof of the induction and hence the proof of the theorem.
\end{proof}

\bibliographystyle{amsalpha}
\bibliography{biblio.bib}

\def\cprime{$'$}
\providecommand{\bysame}{\leavevmode\hbox to3em{\hrulefill}\thinspace}
\providecommand{\MR}{\relax\ifhmode\unskip\space\fi MR }
\providecommand{\MRhref}[2]{%
  \href{http://www.ams.org/mathscinet-getitem?mr=#1}{#2}
}
\providecommand{\href}[2]{#2}
\begin{thebibliography}{CdGW15b}

\bibitem[ABW23]{ABW}
Amol Aggarwal, Alexei Borodin, and Michael Wheeler, \emph{Colored fermionic vertex models and symmetric functions}, Commun. Amer. Math. Soc. \textbf{3} (2023), 400--630. \MR{4628347}

\bibitem[AMW24]{AMW}
A.~Ayyer, J.~Martin, and L.~Williams, \emph{The inhomogeneous $t$-{P}ush{TASEP} and {M}acdonald polynomials}, arXiv:2403.10485, 2024.

\bibitem[BDD23]{BenDaliDolega2023}
Houcine Ben~Dali and Maciej Do{\l}{\k{e}}ga, \emph{Positive formula for {J}ack polynomials, {J}ack characters and proof of {L}assalle's conjecture}, Preprint arXiv:2305.07966, 2023.

\bibitem[BDW]{BenDaliWilliams2}
Houcine Ben~Dali and Lauren Williams, \emph{A probabilistic model for interpolation polynomials at $q=1$}, in preparation.

\bibitem[BDW25]{BDWv1}
\bysame, \emph{A combinatorial formula for interpolation {M}acdonald polynomials}, https://arxiv.org/pdf/2510.02587v1, 2025.

\bibitem[BG24]{BeliakovaGorsky2024}
Anna Beliakova and Eugene Gorsky, \emph{Cyclotomic expansions for {$\mathfrak {gl}_N$} link invariants via interpolation {M}acdonald polynomials}, Selecta Math. (N.S.) \textbf{30} (2024), no.~5, Paper No. 101, 60. \MR{4816143}

\bibitem[CdGW15a]{cantini-degier-wheeler-2015}
Luigi Cantini, Jan de~Gier, and Michael Wheeler, \emph{Matrix product formula for {M}acdonald polynomials}, J. Phys. A \textbf{48} (2015), no.~38, 384001.

\bibitem[CdGW15b]{CantinideGierWheeler2015}
Luigi Cantini, Jan de~Gier, and Michael Wheeler, \emph{Matrix product formula for {M}acdonald polynomials}, J. Phys. A \textbf{48} (2015), no.~38, 384001, 25. \MR{3400909}

\bibitem[CdGW20]{chen-degier-wheeler-2020}
Zeying Chen, Jan de~Gier, and Michael Wheeler, \emph{Integrable stochastic dualities and the deformed {K}nizhnik-{Z}amolodchikov equation}, Internat. Math. Res. Notices \textbf{2020} (2020), no.~19, 5872--5925.

\bibitem[Che95]{cherednik-1995b}
Ivan Cherednik, \emph{Nonsymmetric {M}acdonald polynomials}, Internat. Math. Res. Notices \textbf{1995} (1995), no.~10, 483--515. \MR{1358032}

\bibitem[CMW22]{CorteelMandelshtamWilliams2022}
Sylvie Corteel, Olya Mandelshtam, and Lauren Williams, \emph{From multiline queues to {M}acdonald polynomials via the exclusion process}, Amer. J. Math. \textbf{144} (2022), no.~2, 395--436. \MR{4401508}

\bibitem[CNO14]{CNO}
Erik Carlsson, Nikita Nekrasov, and Andrei Okounkov, \emph{Five dimensional gauge theories and vertex operators}, Mosc. Math. J. \textbf{14} (2014), no.~1, 39--61, 170. \MR{3221946}

\bibitem[Do{\l}17]{Dolega2017}
Maciej Do{\l}{\k{e}}ga, \emph{Strong factorization property of {M}acdonald polynomials and higher-order {M}acdonald's positivity conjecture}, J. Algebraic Combin. \textbf{46} (2017), no.~1, 135--163. \MR{3666415}

\bibitem[Fer11]{ferreira-2011}
Jeffrey~Paul Ferreira, \emph{Row-strict {Q}uasisymmetric {S}chur {F}unctions, {C}haracterizations of {D}emazure {A}toms, and {P}ermuted {B}asement {N}onsymmetric {M}acdonald {P}olynomials}, ProQuest LLC, Ann Arbor, MI, 2011, PhD thesis, University of California, Davis. \MR{3022565}

\bibitem[Hai01]{Haiman2001}
Mark Haiman, \emph{Hilbert schemes, polygraphs and the {M}acdonald positivity conjecture}, J. Amer. Math. Soc. \textbf{14} (2001), no.~4, 941--1006. \MR{1839919}

\bibitem[HHL05]{HaglundHaimanLoehr2005}
Jim Haglund, Mark Haiman, and Nicholas Loehr, \emph{A combinatorial formula for {M}acdonald polynomials}, J. Amer. Math. Soc. \textbf{18} (2005), no.~3, 735--761. \MR{2138143}

\bibitem[HHL08]{HHL3}
J.~Haglund, M.~Haiman, and N.~Loehr, \emph{A combinatorial formula for nonsymmetric {M}acdonald polynomials}, Amer. J. Math. \textbf{130} (2008), no.~2, 359--383. \MR{2405160 (2009f:05262)}

\bibitem[Kno97]{Knop1997b}
Friedrich Knop, \emph{Symmetric and non-symmetric quantum {Capelli} polynomials}, Comment. Math. Helv. \textbf{72} (1997), no.~1, 84--100.

\bibitem[KNTZ20]{Kameyama}
Masaya Kameyama, Satoshi Nawata, Runkai Tao, and Hao~Derrick Zhang, \emph{Cyclotomic expansions of {HOMFLY}-{PT} colored by rectangular {Y}oung diagrams}, Lett. Math. Phys. \textbf{110} (2020), no.~10, 2573--2583. \MR{4146944}

\bibitem[Mac88]{Macdonald1988}
I.~G. Macdonald, \emph{A new class of symmetric functions}, Publ. IRMA Strasbourg \textbf{372} (1988), 131--171.

\bibitem[Mac95]{Macdonald1995}
\bysame, \emph{{Symmetric functions and {H}all polynomials}}, second ed., {Oxford Mathematical Monographs}, The Clarendon Press Oxford University Press, New York, 1995, With contributions by A. Zelevinsky, Oxford Science Publications. \MR{1354144}

\bibitem[Man]{MandelshtamTBD}
Olya Mandelshtam, \emph{On {ASEP} polynomials}, in preparation.

\bibitem[Mar20]{Martin}
James~B. Martin, \emph{Stationary distributions of the multi-type {ASEP}}, Electron. J. Probab. \textbf{25} (2020), Paper No. 43, 41. \MR{4089793}

\bibitem[NSS23]{NakviSahiSergel2023}
Yusra Naqvi, Siddhartha Sahi, and Emily Sergel, \emph{Interpolation polynomials, bar monomials, and their positivity}, Int. Math. Res. Not. IMRN (2023), no.~8, 6809--6844. \MR{4574389}

\bibitem[Oko98]{Okounkov1998}
Andrei Okounkov, \emph{{({S}hifted) {M}acdonald polynomials: {$q$}-integral representation and combinatorial formula}}, Compositio Math. \textbf{112} (1998), no.~2, 147--182. \MR{MR1626029 (99h:05120)}

\bibitem[Ols19]{Olshanski2019}
Grigori Olshanski, \emph{Interpolation {M}acdonald polynomials and {C}auchy-type identities}, J. Combin. Theory Ser. A \textbf{162} (2019), 65--117. \MR{3873872}

\bibitem[Sah96]{Sahi1996}
Siddhartha Sahi, \emph{Interpolation, integrality, and a generalization of {M}acdonald's polynomials}, Internat. Math. Res. Notices \textbf{1996} (1996), no.~10, 457--471.

\end{thebibliography}
   
\end{document}